\documentclass[12pt]{article}

\usepackage{amssymb,a4}
\usepackage{amsmath,amsfonts,amssymb,amsthm}

\newcommand{\Z}{{\mathbb Z}}

\newcommand{\C}{{\mathbb C}}

\newcommand{\N}{{\mathbb N}}

\newcommand{\E}{{\mathcal E}}

\def\<{\langle}
\def\>{\rangle}

\def\End{\mathrm{End}}

\def\Res{\mathrm{Res}}

\newtheorem{thm}{Theorem}[section]
\newtheorem{prop}[thm]{Proposition}
\newtheorem{lem}[thm]{Lemma}
\newtheorem{cor}[thm]{Corollary}
\newtheorem{rmk}[thm]{Remark}
\newtheorem{definition}[thm]{Definition}

\begin{document}

\makeatletter \@addtoreset{equation}{section}
\def\theequation{\thesection.\arabic{equation}}
\makeatother \makeatletter

\begin{center}
{\Large \bf   Certain Clifford-like algebra and quantum vertex algebras}
\end{center}

\begin{center}
{Haisheng Li$^{a}$\footnote{Partially supported by
 China NSF grant (No.11471268)},
Shaobin Tan$^{b}$\footnote{Partially supported by China NSF grant (No.11471268)}
and Qing Wang$^{b}$\footnote{Partially supported by
 China NSF grant (No.11371024), Natural Science Foundation of Fujian Province
(No. 2013J01018) and Fundamental Research Funds for the Central
University (No.2013121001).}\\
$\mbox{}^{a}$Department of Mathematical Sciences\\
Rutgers University, Camden, NJ 08102, USA\\
$\mbox{}^{b}$School of Mathematical Sciences, Xiamen University,
Xiamen 361005, China}
\end{center}

\begin{abstract}
In this paper, we study in the context of quantum vertex algebras a certain Clifford-like algebra introduced by Jing and Nie.
We establish bases of PBW type and classify its $\N$-graded irreducible modules by using a notion of Verma module.
On the other hand, we introduce a new algebra, a twin of the original algebra.
Using this new algebra we construct a quantum vertex algebra and  we associate $\N$-graded modules for Jing-Nie's
Clifford-like algebra with $\phi$-coordinated modules for the quantum vertex algebra.
We also show that the adjoint module for the quantum vertex algebra is irreducible.
\end{abstract}

\section{Introduction}

In their study of vertex operator realization of symplectic and orthogonal Schur functions, as the main ingredient
Jing and Nie introduced in \cite{JN} certain vertex operators $Y(z)$ and $Y^{*}(z)$ on the algebra of polynomials in infinitely
 many variables. They proved that  $Y(z)$ and $Y^{*}(z)$ satisfy the following relations
 \begin{eqnarray*}
&&Y(z)Y(w)+\frac{w}{z}Y(w)Y(z)=0, \label{eq:3.4}\\
&&Y^{\ast}(z)Y^{\ast}(w)+\frac{w}{z}Y^{\ast}(w)Y^{\ast}(z)=0, \label{eq:3.5}\\
&&Y(z)Y^{\ast}(w)+\frac{z}{w}Y^{\ast}(w)Y(z)=\delta\left(\frac{z}{w}\right). \label{eq:3.6}
\end{eqnarray*}
 In terms of components, these relations read as
 $$Y_mY_n+Y_{n+1}Y_{m-1}=0, \   \   \
Y_m^{\ast}Y_n^{\ast}+Y_{n+1}^{\ast}Y_{m-1}^{\ast}=0, \   \  \
Y_mY_n^{\ast}+Y_{n-1}^{\ast}Y_{m+1}=\delta_{m+n,0} $$
for $m,n\in \Z$, where
$$Y(z)=\sum_{n\in \Z}Y_{n}z^{-n},\   \   \   \  Y^{*}(z)=\sum_{n\in \Z}Y^{*}_{n}z^{-n}.$$
This naturally gives rise to an infinite-dimensional Clifford-like algebra, which is the main object of our study
in this current paper. This is one of what are called Zamolodchikov-Faddeev algebras in mathematical physics.

In this paper, we study this Clifford-like algebra from a vertex-algebra point of view
and our main objective is to determine  its irreducible modules of highest weight type and to associate them
with vertex algebras or more general quantum vertex algebras.
We first define an abstract associative algebra denoted by $\mathcal{A}$
 to be the algebra (with identity over $\C$) with generators $Y_{n},\ Y^{*}_{n}$ for $n\in \Z$, subject to the above relations.
As for the algebraic structure, we establish a basis of  PBW type and we classify its irreducible $\N$-graded modules,
noticing that $\mathcal{A}$ is naturally $\Z$-graded. On the other hand,
 we introduce another closely related associative algebra $\tilde{\mathcal{A}}$, and
 construct a quantum vertex algebra $V_{\tilde{\mathcal{A}}}$ in the sense of \cite{Li-nonlocal}, which is
based on a distinguished $\tilde{\mathcal{A}}$-module. Finally,  we
establish a canonical association of  $\N$-graded $\mathcal{A}$-modules with $\phi$-coordinated
$V_{\tilde{\mathcal{A}}}$-modules in the sense of \cite{Li-phi}.

Now, we give a more detailed account of the contents of this paper.
First, we determine the algebraic structure of algebra $\mathcal{A}$.
To this end, we establish a base of PBW type by using the Diamond Lemma.
We also use an infinite-dimensional Clifford algebra to obtain another base of PBW type. As $\mathcal{A}$
is a $\Z$-graded algebra with $\deg Y_{n}=-n=\deg Y^{*}_{n}$ for $n\in \Z$, we study its $\N$-graded modules.
Note that the sum of $\mathcal{A}_{n}\mathcal{A}_{-n}$ for $n\ge 1$ is a (two-sided) ideal of
 the degree zero subalgebra $\mathcal{A}_{0}$,
which we denote by $I$. Set $\bar{\mathcal{A}}_{0}=\mathcal{A}_{0}/I$.
We show that irreducible $\N$-graded $\mathcal{A}$-modules
one-to-one correspond to irreducible $\bar{\mathcal{A}}_{0}$-modules.
Furthermore,  we prove that $\bar{\mathcal{A}}_{0}$ is isomorphic to the group algebra of $\Z$ and
all irreducible $\bar{\mathcal{A}}_{0}$-modules are $1$-dimensional parametrized by nonzero complex numbers.
In this way, we have classified irreducible $\N$-graded modules for $\mathcal{A}$. Note that Jing and Nie gave an explicit realization
of $\mathcal{A}$. In this paper, we show that this $\mathcal{A}$-module is an irreducible $\N$-graded module and there exists a non-degenerate symmetric bilinear form which is invariant with respect to a specific involution of $\mathcal{A}$.

To associate (quantum) vertex algebras to $\mathcal{A}$, we start with its generating functions
$Y(z)$ and $Y^{*}(z)$. From the defining relations,
we see that  $Y(z)$ and $Y^{*}(z)$ on $\N$-graded $\mathcal{A}$-modules do not satisfy
the locality condition which is essential to the construction of vertex algebras and modules (see \cite{Li-local}).
Nevertheless, they are $S_{trig}$-local in the sense of \cite{Li-phi}, where a theory of what were called $\phi$-coordinated modules
for (weak) quantum vertex algebras in the sense of \cite{Li-nonlocal} was developed.
It was proved in  \cite{Li-phi} that
any set of $S_{trig}$-local formal vertex operators on a vector space $W$  generate  a weak quantum vertex algebra
with $W$ as a faithful $\phi$-coordinated module. It is this result that we shall use to associate quantum vertex algebras to
algebra $\mathcal{A}$.

To obtain a concrete association of quantum vertex algebras to $\mathcal{A}$,
as the main ingredient we introduce another algebra which we denote by $\tilde{\mathcal{A}}$.
By definition, algebra $\tilde{\mathcal{A}}$  is generated by
$\tilde{Y}_{n},\  \tilde{Y}^{*}_{n}$ with $n\in \Z$, subject to relations
 \begin{eqnarray*}
&&\tilde{Y}(z)\tilde{Y}(w)+e^{w-z}\tilde{Y}(w)\tilde{Y}(z)=0, \label{eq:3.4}\\
&&\tilde{Y}^{\ast}(z)\tilde{Y}^{\ast}(w)+e^{w-z}\tilde{Y}^{\ast}(w)\tilde{Y}^{\ast}(z)=0, \label{eq:3.5}\\
&&\tilde{Y}(z)\tilde{Y}^{\ast}(w)+e^{z-w}\tilde{Y}^{\ast}(w)\tilde{Y}(z)=w^{-1}\delta\left(\frac{z}{w}\right), \label{eq:3.6}
\end{eqnarray*}
where $\tilde{Y}(z)=\sum_{n\in \Z}\tilde{Y}_{n}z^{-n-1}$ and $\tilde{Y}^{*}(z)=\sum_{n\in \Z}\tilde{Y}^{*}_{n}z^{-n-1}$.
(Note that this is also a Zamolodchikov-Faddeev algebra.)
By using an infinite-dimensional Clifford algebra, we give a realization of $\tilde{\mathcal{A}}$.
It is important to note that the generating functions $\tilde{Y}(x)$ and $\tilde{Y}^{*}(x)$ of $\tilde{\mathcal{A}}$ form an
$\mathcal{S}$-local set in the sense of \cite{Li-nonlocal}.
Let $V_{\tilde{\mathcal{A}}}$ be the $\tilde{\mathcal{A}}$-module generated by a distinguished vector ${\bf 1}$, subject to relations
$$\tilde{Y}_{n}{\bf 1}=0=\tilde{Y}^{*}_{n}{\bf 1}\   \   \   \mbox{ for  }n\ge 0.$$
Set
$\tilde{a}=\tilde{Y}_{-1}{\bf 1},    \  \tilde{b}=\tilde{Y}^{*}_{-1}{\bf 1}\in V_{\tilde{\mathcal{A}}}.$
By using a result of \cite{Li-nonlocal}, we show that there is a weak quantum vertex algebra structure
on $V_{\tilde{\mathcal{A}}}$, which is uniquely determined by the condition that  ${\bf 1}$ is the vacuum vector
and
$$Y(\tilde{a},x)=\tilde{Y}(x),\    \    \    \  Y(\tilde{b},x)=\tilde{Y}^{*}(x) .$$
On the other hand, by using a certain increasing filtration and a result of \cite{KL},
we show that $V_{\tilde{\mathcal{A}}}$ is an irreducible $\tilde{\mathcal{A}}$-module.
It then follows that $V_{\tilde{\mathcal{A}}}$ is a (non-degenerate) quantum vertex algebra. We achieve our main goal
by showing that a restricted $\mathcal{A}$-module structure on a vector space $W$ exactly amounts to a $\phi$-coordinated
$V_{\tilde{\mathcal{A}}}$-module structure $Y_{W}(\cdot,x)$ with
$$Y_{W}(\tilde{a},x)=Y(x),\    \    \    \   Y_{W}(\tilde{b},x)=Y^{*}(x).$$

This paper is organized as follows: In Section 2, we first define the algebra $\mathcal{A}$, and
then we give a base of PBW type and classify its irreducible $\N$-graded modules.
In Section 3, we introduce the algebra $\tilde{\mathcal{A}}$, construct a quantum vertex algebra $V_{\tilde{\mathcal{A}}}$, and show that  $V_{\tilde{\mathcal{A}}}$ as an $\tilde{\mathcal{A}}$-module is irreducible. Furthermore, we give a canonical association of $\phi$-coordinated modules for $V_{\tilde{\mathcal{A}}}$ to restricted $\mathcal{A}$-modules.

\section{Preliminaries}
%\label{Sect:V(k,0)}
%\def\theequation{2.\arabic{equation}}
%\setcounter{equation}{0}

In this section,  we recall from \cite{Li-nonlocal} and \cite{Li-phi} some basic notations and results on quantum vertex algebras
and their $\phi$-coordinated modules, including the conceptual construction of (weak) quantum vertex algebras and
 $\phi$-coordinated modules.

Throughout this paper, $\N$ denotes the set of nonnegative integers,
 $\mathbb{C}^{\times}$ denotes the multiplicative group of nonzero complex numbers
 (while $\C$ denotes the complex number field), and
the symbols $x,y,x_{1},x_{2},\dots $ denote mutually commuting independent formal variables.
All vector spaces in this paper are
considered to be over  $\mathbb{C}$. For a vector space $U$, $U((x))$ is the vector space of lower
truncated integer power series in $x$ with coefficients
 in $U$, $U[[x]]$ is the vector space of nonnegative integer
 power series in $x$ with coefficients in $U$, and
$U[[x,x^{-1}]]$ is the vector space of doubly infinite integer power series in $x$ with coefficients in $U$.

We begin by recalling the definitions of nonlocal vertex algebras and modules (see \cite{Li-nonlocal}, \cite{Li-g1}; cf. \cite{B}, \cite{BK}).

\begin{definition}\label{defnlva}
{\em A {\em nonlocal vertex algebra} is a vector space $V$ equipped with a linear map
      \begin{eqnarray*}
             Y(\cdot,x) :&&V\longrightarrow \mathrm{Hom}(V,V((x)))\subset (\mathrm{EndV})[[x,x^{-1}]]\\
              &&v\longmapsto Y(v,x)=\sum_{n\in\mathbb{Z}}v_{n}x^{-n-1}\ \ (\mbox{where }v_{n}\in \End V)
      \end{eqnarray*}
              and equipped with a distinguished vector $\textbf{1}\in V$, called the {\em vacuum vector},
              satisfying the conditions that
              $$Y(\textbf{1},x)v=v,$$
              $$Y(v,x)\textbf{1}\in V[[x]] \;\;\mbox{and}\;\; \lim_{x\rightarrow  0}Y(v,x)\textbf{1} = v,$$
              and  that for $u,v,w\in V$, there exists a nonnegative integer $l$ such that
       \begin{eqnarray*}
             (x_0+x_2)^{l}Y(u,x_{0}+x_2)Y(v,x_{2})w=(x_0+x_2)^{l}Y(Y(u,x_0)v,x_2)w.
         \end{eqnarray*}}
\end{definition}

\begin{definition}\label{defmodule}
{\em Let $V$ be a nonlocal vertex algebra. A {\em $V$-module} is a vector space
                  $W$ equipped with a linear map
   \begin{eqnarray*}
     Y_{W}(\cdot,x) :&&V\longrightarrow \mathrm{Hom}(W,W((x)))\subset (\mathrm{EndW})[[x,x^{-1}]]\\
              && v\longmapsto Y_{W}(v,x)=\sum_{n\in\mathbb{Z}}v_{n}x^{-n-1}\ \ (\mbox{where }v_{n}\in \End W),
   \end{eqnarray*}
              satisfying the conditions that
                   $$ Y_{W}(1,x) = 1_{W}\ \ (\mbox{the identity operator on }W)$$
                    and  that for any $u,v\in V$, $w\in W$, there exists a nonnegative integer $l$
                    such that
                    \begin{eqnarray*}
              &&(x_0+x_2)^{l}Y_{W}(u,x_{0}+x_2)Y_{W}(v,x_{2})w=(x_0+x_2)^{l}Y_{W}(Y(u,x_0)v,x_2)w.
                         \end{eqnarray*}}
\end{definition}

The last condition in Definitions \ref{defnlva} and \ref{defmodule} is usually referred to as {\em weak associativity}.

We recall from \cite{Li-nonlocal} (cf. \cite{Li-const}, \cite{Li-phi}) the notion of weak quantum vertex algebra.

\begin{definition}
{\em A {\em weak quantum vertex algebra} is a vector space $V$ equipped with a linear map
             $Y(\cdot,x) :V\rightarrow \mathrm{Hom}(V,V((x)))$ and a vector $\textbf{1}\in V$,
              satisfying the conditions that for $v\in V$,
              $$Y(\textbf{1},x)v=v,$$
              $$Y(v,x)\textbf{1}\in V[[x]] \;\;\mbox{and}\;\; \lim_{x\rightarrow  0}Y(v,x)\textbf{1} = v,$$
 and that for $u,v\in V$, there exists $\sum_{i=1}^{r}v^{(i)}\otimes u^{(i)}\otimes f_{i}(x)\in V\otimes V\otimes {\mathbb{C}}((x))$
 such that
              \begin{eqnarray}\label{eSjacobi}
              &&x_{0}^{-1}\delta\left(\frac{x_{1}-x_{2}}{x_{0}}\right)Y(u,x_{1})Y(v,x_{2}) -
                         x_{0}^{-1}\delta\left(\frac{x_{2}-x_{1}}{-x_{0}}\right)\sum_{i=1}^{r}f_{i}(-x_{0})Y(v^{(i)},x_{2})Y(u^{(i)},x_{1})\nonumber\\
               &&\hspace{2cm}= x_{2}^{-1}\delta\left(\frac{x_{1}-x_{0}}{x_{2}}\right)Y(Y(u,x_{0})v,x_{2}). \label{eq:2.1}
              \end{eqnarray}}
\end{definition}

Note that a weak quantum vertex algebra is automatically a nonlocal vertex algebra as the $\mathcal{S}$-Jacobi identity
(\ref{eSjacobi}) implies the weak associativity. It is clear that the notion of weak quantum vertex algebra generalizes that of vertex algebra and vertex super-algebra.

\begin{rmk}
{\em Let $V$ be a weak quantum vertex algebra.
A {\em $V$-module} is defined to be a module for $V$ viewed as a nonlocal vertex algebra.
Let $(W,Y_{W})$ be any $V$-module.
From \cite{Li-g1}, for $u,v\in V$, whenever (\ref{eq:2.1}) holds, we have
 \begin{eqnarray*}
              &&x_{0}^{-1}\delta\left(\frac{x_{1}-x_{2}}{x_{0}}\right)Y_{W}(u,x_{1})Y_{W}(v,x_{2}) -
                         x_{0}^{-1}\delta\left(\frac{x_{2}-x_{1}}{-x_{0}}\right)\sum_{i=1}^{r}f_{i}(-x_{0})Y_{W}(v^{(i)},x_{2})Y(u^{(i)},x_{1})\\
               &&\hspace{2cm}= x_{2}^{-1}\delta\left(\frac{x_{1}-x_{0}}{x_{2}}\right)Y_{W}(Y(u,x_{0})v,x_{2}).
              \end{eqnarray*} }
\end{rmk}

Recall that a {\em rational quantum Yang-Baxter operator} on a vector space $U$ is a linear map
 $$S(x):\   U\otimes U\rightarrow U\otimes U\otimes \mathbb{C}((x)),$$
 satisfying $$S^{12}(x)S^{13}(x+z)S^{23}(z)=S^{23}(z)S^{13}(x+z)S^{12}(x)$$
 (the {\em quantum Yang-Baxter equation}), where for $1\leq i<j\leq 3$,
 $$S^{ij}(x): U\otimes U\otimes U\rightarrow U\otimes U\otimes U\otimes \mathbb{C}((x))$$
 denotes the canonical extension of $S(x)$. It is said to be {\em unitary} if $S(x)S^{21}(-x)=1,$
 where $S^{21}(x)=\sigma S(x)\sigma$ with $\sigma$ denoting the flip operator on $U\otimes U$.

The following notion of quantum vertex algebra was introduced in \cite{Li-nonlocal} (cf. \cite{EK}):

\begin{definition}
{\em A {\em quantum vertex algebra} is a weak quantum vertex algebra $V$ equipped with a unitary rational quantum Yang-Baxter operator $S(x)$ on $V$ such that for $u,v\in V$, (\ref{eq:2.1}) holds with $\sum_{i=1}^{r}v^{(i)}\otimes u^{(i)}\otimes f_{i}(x)=S(x)(v\otimes u)$.}
\end{definition}

The following notion is due to Etingof and Kazhdan (see \cite{EK}):

\begin{definition}
{\em A nonlocal vertex algebra $V$ is said to be {\em non-degenerate} if for every positive integer $n$, the linear map
$$Z_n:\  \mathbb{C}((x_{1}))\cdots((x_{n}))\otimes V^{\otimes n}\rightarrow V((x_{1}))\cdots((x_{n}))$$ defined by
$$Z_n(f\otimes v^{(1)}\otimes\cdots\otimes v^{(n)})=fY(v^{(1)},x_{1})\cdots Y(v^{(n)},x_{n})\textbf{1}$$
is injective.}
\end{definition}

\begin{rmk}
{\em It was proved in \cite{Li-nonlocal} that every non-degenerate weak quantum vertex algebra is a quantum vertex algebra
with a uniquely determined rational quantum Yang-Baxter operator.  Furthermore, it was proved therein
that if $V$ is of countable dimension and if $V$ as a (left) $V$-module is irreducible, then $V$ is non-degenerate. }
\end{rmk}

Next, we recall from \cite{Li-phi} and \cite{Li6} some basic notions and results in the theory of $\phi$-coordinated modules for weak quantum vertex algebras. In this theory, $\phi$ stands for the formal series $\phi(x,z)=xe^{z}\in \mathbb{C}[[x,z]]$, which is what is called therein an associate of the $1$-dimensional additive formal group (law).

\begin{definition}
{\em Let $V$ be a weak quantum vertex algebra. A {\em $\phi$-coordinated $V$-module} is a vector space
                  $W$ equipped with a linear map
                  $$Y_{W}(\cdot,x) :\  V\longrightarrow \mathrm{Hom}(W,W((x)))\subset (\mathrm{EndW})[[x,x^{-1}]],$$
              satisfying the conditions that $Y_{W}(1,x) = 1_{W}$
                    and that for any $u,v\in V$, there exists $k\in \mathbb{N}$
                    such that
                    \begin{eqnarray*}
              &&(x_1-x_2)^{k}Y_{W}(u,x_{1})Y_{W}(v,x_{2})\in \mathrm{Hom}(W,W((x_{1},x_{2})))
                         \end{eqnarray*}
                         and
                    \begin{eqnarray*}
              &&(x_{2}e^{z}-x_2)^{k}Y_{W}(Y(u,z)v,x_{2})=((x_1-x_2)^{k}Y_{W}(u,x_{1})Y_{W}(v,x_{2})|_{x_1=x_{2}e^{z}}.
                         \end{eqnarray*}}
\end{definition}

The following result was obtained in \cite{Li-phi}:

\begin{prop}
Let $V$ be a weak quantum vertex algebra and let $(W,Y_{W})$ be a $\phi$-coordinated $V$-module. Let $u,v\in V$.
Suppose that $u^{(i)},\ v^{(i)}\in V,\ f_i(x)\in \C(x)\ (1\le i\le r)$ such that
\begin{eqnarray}
(x_1-x_2)^{k}Y(u,x_1)Y(v,x_2)=(x_1-x_2)^{k}\sum_{i=1}^{r}f_i(e^{x_2-x_1})Y(v^{(i)},x_2)Y(u^{(i)},x_1)
\end{eqnarray}
on $V$ for some nonnegative integer $k$. Then
\begin{eqnarray}
&&Y_{W}(u,x_1)Y_W(v,x_2)-\sum_{i=1}^{r}f_{i}(x_2/x_1)Y_{W}(v^{(i)},x_2)Y_{W}(u^{(i)},x_1)\nonumber\\
&=&\sum_{n\ge 0}\frac{1}{n!}Y_{W}(u_nv,x_2)\left(x_2\frac{\partial}{\partial x_2}\right)^{n}\delta\left(\frac{x_2}{x_1}\right).
\end{eqnarray}
\end{prop}

 Let $W$ be a general vector space. Set
\begin{eqnarray*}
\mathcal{E}(W)=\mbox{Hom}(W , W((x)))\subset(\mbox{EndW})[[x,x^{-1}]].
\end{eqnarray*}
The identity operator on $W$, denoted by $\textbf{1}_{W}$, is a special
element of $\mathcal{E}(W).$

\begin{definition} {\em A subset (subspace) $U$ of $\mathcal{E}(W)$ is said to be $S_{trig}$-local
if for any $a(x), b(x)\in U$, there exist
$$u_{i}(x), v_{i}(x)\in U, \ q_{i}(x)\in \mathbb{C}(x), \ i=1,\ldots,r,$$
where $\mathbb{C}(x)$ denotes the field of rational functions, such that
\begin{eqnarray}
(x_1-x_2)^{k}a(x_{1})b(x_{2})=(x_1-x_2)^{k}\sum_{i=1}^{r}\iota_{x_2,x_1}(q_{i}(x_{1}/x_{2}))u_{i}(x_2)v_{i}(x_1)\label{eq:2.2}
\end{eqnarray}
for some $k\in \mathbb{N}$, where
$$\iota_{x_{1},x_{2}}: \  \C(x_{1},x_{2})\rightarrow \C((x_{1}))((x_{2}))$$
is the canonical extension of the ring embedding of $\C[x_{1},x_{2}]$ into the field $\C((x_{1}))((x_{2}))$.}
\end{definition}

Let $W$ be a vector space as before. Let $U$ be any $S_{trig}$-local subset of $\mathcal{E}(W)$ and let $a(x),b(x)\in U$.
Notice that the relation (\ref{eq:2.2}) implies
 \begin{eqnarray}
 (x_{1}-x_2)^{k}a(x_{1})b(x_{2})\in \mbox{Hom}(W, W((x_{1},x_{2}))).     \label{eq:2.3}
 \end{eqnarray}
 Define $a(x)_{n}^{e}b(x)\in\mbox{(End W)}[[x,x^{-1}]]$ for $n\in \Z$
                  in terms of generating function
                 \begin{eqnarray}Y_{\mathcal{E}}^{e}(a(x),z)b(x)=
                  \sum_{n\in\mathbb{Z}}(a(x)_{n}^{e}b(x))z^{-n-1}   \label{eq:2.4}\end{eqnarray}
                  by
                 \begin{eqnarray}Y_{\mathcal{E}}^{e}(a(x),z)b(x)=
                  (xe^{z}-x)^{-k}((x_1-x)^{k}a(x_{1})b(x))|_{x_{1}=xe^{z}},  \label{eq:2.5}\end{eqnarray}
                 where $k$ is any nonnegative integer such that (\ref{eq:2.3}) holds.

Let $U$ be an $S_{trig}$-local subspace of $\mathcal{E}(W)$. We say
 $U$ is {\em $Y_{\mathcal{E}}^{e}$-closed} if
           \begin{eqnarray*}
                 a(x)_{n}^{e}b(x)\in U
                  \    \    \mbox{ for all }a(x),b(x)\in U,\ n\in\mathbb{Z}.
            \end{eqnarray*}

The following result was obtained in \cite{Li-phi} (Theorem 5.4):

\begin{thm}\label{thm2.1}
Let $U$ be any $S_{trig}$-local subset of $\mathcal{E}(W)$. Then there exists a $Y_{\mathcal{E}}^{e}$-closed
$S_{trig}$-local subspace of $\mathcal{E}(W)$, which contains $U$ and $\textbf{1}_{W}$.
Denote by $\langle U\rangle_{e}$ the smallest such subspace. Then
$(\langle U\rangle_{e},Y_{\mathcal{E}}^{e},1_{W})$ carries the structure of a weak quantum vertex algebra and  $W$ is a
  $\phi$-coordinated $\langle U\rangle_{e}$-module with
 $Y_{W}(a(x),z) = a(z)$
  for $a(x)\in\langle U\rangle_{e}.$
\end{thm}

\section{ Clifford-like algebra $\mathcal{A}$}

In this section, we first define the Clifford-like algebra $\mathcal{A}$
 and we then establish a PBW type basis and classify irreducible $\N$-graded $\mathcal{A}$-modules by using
a notion of Verma module. We also prove that the $\mathcal{A}$-module given by Jing-Nie is an irreducible $\N$-graded module
and there exists a non-degenerate symmetric bilinear form which is invariant in a certain sense.

First, we recall from  \cite{JN} the Clifford-like algebra.

\begin{definition}
{\em The algebra $\mathcal{A}$ is defined to be the associative algebra with identity over $\mathbb{C}$
with generators
$Y_n,\  Y_{n}^{\ast}\ (n\in \mathbb{Z})$, subject to relations
\begin{eqnarray}
&&Y_mY_n+Y_{n+1}Y_{m-1}=0, \label{eq:3.1}\\
&&Y_m^{\ast}Y_n^{\ast}+Y_{n+1}^{\ast}Y_{m-1}^{\ast}=0, \label{eq:3.2}\\
&&Y_mY_n^{\ast}+Y_{n-1}^{\ast}Y_{m+1}=\delta_{m+n,0} \label{eq:3.3}
\end{eqnarray}
for $m,n\in \mathbb{Z}$.}
\end{definition}

Note that for a certain purpose, we made a notational adjustment in the original definition of  \cite{JN};
the generator $Y_{n}^{\ast}$ here corresponds to $Y_{-n}^{\ast}$ therein.

It can be readily seen that $\mathcal{A}$ is a $\Z$-graded algebra with $\deg Y_{n}=-n=\deg Y^{*}_{n}$ for $n\in \Z$.
For $n\in \Z$, let $\mathcal{A}_{n}$ denote the homogeneous subspace of degree $n$.
It follows that $\mathcal{A}$ admits a derivation $d$ such that
\begin{eqnarray}
d(Y_{n})=-nY_{n}, \  \  d(Y^{*}_{n})=-nY^{*}_{n} \   \  \   \mbox{ for }n\in \Z.
\end{eqnarray}
%Set
%\begin{eqnarray}
 %\bar{\mathcal{A}}=\mathcal{A} \sharp \C[d],
%\end{eqnarray}
%the smash product.

Form generating functions
\begin{eqnarray}
Y(z)=\sum_{n\in \mathbb{Z}}Y_{n}z^{-n}, \  \   \ \
Y^{\ast}(z)=\sum_{n\in \mathbb{Z}}Y^{\ast}_{n}z^{-n}.
\end{eqnarray}
Then the defining relations (\ref{eq:3.1})-(\ref{eq:3.3}) can be written as
\begin{eqnarray}
&&Y(z)Y(w)+\frac{w}{z}Y(w)Y(z)=0, \label{eq:3.4}\\
&&Y^{\ast}(z)Y^{\ast}(w)+\frac{w}{z}Y^{\ast}(w)Y^{\ast}(z)=0, \label{eq:3.5}\\
&&Y(z)Y^{\ast}(w)+\frac{z}{w}Y^{\ast}(w)Y(z)=\delta\left(\frac{z}{w}\right). \label{eq:3.6}
\end{eqnarray}

As the first result, we have:

\begin{prop}\label{Abasis}
The following vectors form a basis of $\mathcal{A}$:
\begin{eqnarray}\label{A-basis}
Y_{m_1}\cdots Y_{m_r}Y^{*}_{n_1}\cdots Y^{*}_{n_s}\cdot 1,
\end{eqnarray}
where $r,s\ge 0,\ m_i,\ n_i\in \Z$, satisfying the condition
$$m_1>m_2+1>\cdots >m_r+(r-1)\  \mbox{ and }\  n_1>n_2+1>\cdots >n_s+(s-1).$$
\end{prop}

\begin{proof} We first prove that $\mathcal{A}$ is linearly spanned by these vectors.
 Let $\mathcal{A}_{Y}$ denote the subalgebra generated by $Y_{n}$ for $n\in \Z$ and let $\mathcal{A}_{Y^{*}}$
denote the subalgebra generated by $Y^{*}_{n}$ for $n\in \Z$.
Due to relation (\ref{eq:3.3}), we have $\mathcal{A}=\mathcal{A}_{Y}\mathcal{A}_{Y^{*}}$.
Then it suffices to prove that $\mathcal{A}_{Y}$ and $\mathcal{A}_{Y^{*}}$ are linearly spanned
by those vectors in (\ref{A-basis}) with $s=0$ and $r=0$, respectively.

Let ${\bf m}=(m_1,m_2,\dots,m_r)$ be an ordered $r$-tuple of integers. We write $Y_{\bf m}$ for $Y_{m_1}\cdots Y_{m_r}$.
For each pair $(i,j)$ with $1\le i<j\le r$, we assign a reverse number $1$ if $m_i<m_j+(j-i)$ and $0$ otherwise.
Then define the total reverse number
 $N({\bf m})$ to be the sum of the reverse numbers for all the pairs $(i,j)$. We see that $N({\bf m})=0$ if and only if
 $m_{i}\ge m_{i+1}+1$ for all $1\le i<r$.
Assume $N({\bf m})>0$. Then there exists $1\le i<r$ such that $m_i<m_{i+1}+1$. We have
$$Y_{\bf m}=Y_{m_1}\cdots Y_{m_{i-1}}(Y_{m_i}Y_{m_{i+1}})Y_{m_{i+2}}\cdots Y_{m_r}
=-Y_{m_1}\cdots Y_{m_{i-1}}(Y_{m_{i+1}+1}Y_{m_{i}-1})Y_{m_{i+2}}\cdots Y_{m_r},$$
where
$$N({\bf m})=N(m_1,\dots,m_{i-1},m_{i+1}+1,m_i-1,m_{i+2},\dots,m_r)+1.$$
Then it follows from induction (on the reverse number) that $\mathcal{A}_{Y}$ is linearly spanned
by $1$ and $Y_{\bf m}$ where ${\bf m}\in \Z^{r}$ for $r\ge 1$ with $N({\bf m})=0$.
Note that for $m\in \Z$, since $Y_{m}Y_{m-1}+Y_{m}Y_{m-1}=0$, we have  $Y_{m}Y_{m-1}=0$.
Then $\mathcal{A}_{Y}$  is linearly spanned by those vectors in (\ref{A-basis}) with $s=0$.
Similarly,  $\mathcal{A}_{Y^{*}}$  is linearly spanned by those vectors in (\ref{A-basis}) with  $r=0$.
Consequently, $\mathcal{A}$ is linearly spanned by the vectors in (\ref{A-basis}).

To establish the linear independence, we apply the Diamond lemma (see \cite{bergman}).
Notice that the following are all the ambiguities:
$$Y_{m}Y_{n}Y_{k}, \  \  \   Y^{*}_{p}Y_{m}Y_{n},\   \  \   Y^{*}_{m}Y^{*}_{n}Y_{p},\   \  \  Y^{*}_{m}Y^{*}_{n}Y^{*}_{k},$$
where $m,n,k,p\in \Z$ with $m\le n+1\le k+2$.
 We have
$$(Y_{m}Y_{n})Y_{k}=-Y_{n+1}Y_{m-1}Y_{k}=Y_{n+1}Y_{k+1}Y_{m-2}=-Y_{k+2}Y_{n}Y_{m-2},$$
$$Y_{m}(Y_{n}Y_{k})=-Y_{m}Y_{k+1}Y_{n-1}=Y_{k+2}Y_{m-1}Y_{n-1}=-Y_{k+2}Y_{n}Y_{m-2},$$
so the first ambiguity can be resolved.  Similarly, the last ambiguity can be resolved.
For the second ambiguity, we have
\begin{eqnarray*}
&&Y^{*}_{p}(Y_{m}Y_{n})=-Y^{*}_{p}Y_{n+1}Y_{m-1}=Y_{n}Y^{*}_{p+1}Y_{m-1}-\delta_{n+p+1,0}Y_{m-1}\\
&=&-Y_{n}Y_{m-2}Y^{*}_{p+2}+\delta_{p+m,0}Y_{n}-\delta_{n+p+1,0}Y_{m-1},
\end{eqnarray*}
\begin{eqnarray*}
&&(Y^{*}_{p}Y_{m})Y_{n}=-Y_{m-1}(Y^{*}_{p+1}Y_{n})+\delta_{m+p,0}Y_{n}\\
&=&(Y_{m-1}Y_{n-1})Y^{*}_{p+2}-\delta_{n+p+1,0}Y_{m-1}+\delta_{m+p,0}Y_{n}\\
&=&-Y_{n}Y_{m-2}Y^{*}_{p+2}-\delta_{n+p+1,0}Y_{m-1}+\delta_{m+p,0}Y_{n}.
\end{eqnarray*}
Thus the second can be resolved. The third  ambiguity can also be resolved as
\begin{eqnarray*}
&&(Y^{*}_{m}Y^{*}_{n})Y_{p}=-Y^{*}_{n+1}Y^{*}_{m-1}Y_{p}
=Y^{*}_{n+1}Y_{p-1}Y^{*}_{m}-\delta_{p+m-1,0}Y^{*}_{n+1}\\
&&\  \  \  \   =-Y_{p-2}Y^{*}_{n+2}Y^{*}_{m}+\delta_{p+n,0}Y^{*}_{m}-\delta_{p+m-1,0}Y^{*}_{n+1},
\end{eqnarray*}
\begin{eqnarray*}
Y^{*}_{m}(Y^{*}_{n}Y_{p})&=&-Y^{*}_{m}Y_{p-1}Y^{*}_{n+1}+\delta_{p+n,0}Y^{*}_{m}\\
&=&Y_{p-2}Y^{*}_{m+1}Y^{*}_{n+1}-\delta_{p+m-1,0}Y^{*}_{n+1}+\delta_{p+n,0}Y^{*}_{m}\\
&=&-Y_{p-2}Y^{*}_{n+2}Y^{*}_{m}-\delta_{p+m-1,0}Y^{*}_{n+1}+\delta_{p+n,0}Y^{*}_{m}.
\end{eqnarray*}
Then it follows from the diamond lemma that those vectors are linearly independent.
\end{proof}

Next, we recall from \cite{JN} the explicit vertex-operator realization of $\mathcal{A}$.
%[Let $\Lambda =\Lambda_{\Q}$ be the ring of symmetric functions in countably many variables $x_{1},x_{2},\dots$
%over the field $\Q$ of rational numbers. The degree of homogeneous symmetric functions give rise to a natural gradation
%for $\Lambda_{\Q}$:
%$$ \Lambda_{\Q}= \bigoplus_{k\ge 0}\Lambda^{k}_{\Q},$$
%where $\Lambda^{k}_{\Q}$ consists of the homogeneous symmetric functions of degree $k$.]
Let  $\mathcal{H}$ be the infinite-dimensional Heisenberg algebra with base
 $\{ a_{n}\ | \ n\in \Z\backslash \{0\}\}\cup \{c\}$ such that
$[c,\mathcal{H}]=0$ and
$$[a_{m}, a_{n}]=m\delta_{m+n,0}c\   \   \  \mbox{ for }m,n\in \Z\backslash \{0\}.$$
 We know that $\mathcal{H}$  acts irreducibly on
polynomial algebra $\C[x_{1},x_{2},\dots]$, which is denoted by $M(1)$,  with
$$c=1,\   \   \   \  a_{n}=n\frac{\partial}{\partial x_{n}},\   \  \   \  a_{-n}=x_{n}\   \   \   \mbox{ for }n\ge 1.$$
It was proved in \cite{JN} that $M(1)$ is an $\mathcal{A}$-module with
\begin{eqnarray}
 &&Y(z) = \exp\left(\sum_{n=1}^{\infty}\frac{a_{-n}}{n} z^n\right)
 \exp\left(-\sum_{n=1}^{\infty}\frac{a_{n}}{n}(z^{-n}+z^n)\right),\\
&& Y^{*}(z) =(1-z^2)\exp\left(-\sum_{n=1}^{\infty}\frac{a_{-n}}{n} z^n\right)
 \exp\left(\sum_{n=1}^{\infty}\frac{a_{n}}{n}(z^{-n}+z^n)\right).
\end{eqnarray}
%\begin{eqnarray}
% &&S(z) = \exp\left(\sum_{n=1}^{\infty}\frac{a_{-n}}{n} z^n\right)
% \exp\left(-\sum_{n=1}^{\infty}\frac{a_{n}}{n}z^{-n}\right),\\
%&& S^{*}(z) =\exp\left(-\sum_{n=1}^{\infty}\frac{a_{-n}}{n} z^n\right)
% \exp\left(\sum_{n=1}^{\infty}\frac{a_{n}}{n}z^{-n}\right).
%\end{eqnarray}
Let $\mathcal{P}$ denote the set of all partitions, where a partition is a  sequence of
 non-increasing nonnegative integers with only finitely many nonzero terms. For a partition $\lambda$, the number of nonzero terms
 is called the {\em length}, while the sum of all terms is called the {\em weight}, denoted by $|\lambda|$.
For $\lambda=(\lambda_1,\lambda_2,\dots,\lambda_r)\in \mathcal{P}$, set
 \begin{eqnarray}
 Y_{-\lambda}=Y_{-\lambda_1}Y_{-\lambda_2}\cdots Y_{-\lambda_r},\   \   \
 Y^{*}_{-\lambda}=Y^{*}_{-\lambda_1}Y^{*}_{-\lambda_2}\cdots Y^{*}_{-\lambda_r}.
 \end{eqnarray}
 For a partition $\lambda$, denote by $\lambda'$ the dual of $\lambda$. It was proved therein that
\begin{eqnarray}\label{eY=Y*}
Y_{-\lambda}\cdot 1=(-1)^{|\lambda|}Y^{*}_{-\lambda'}\cdot 1
\end{eqnarray}
for $\lambda\in \mathcal{P}$, and  moreover,
$\{ Y_{-\lambda}\cdot 1 \ |\  \lambda\in \mathcal{P}\}$ is a base of $M(1)$.
In this realization, we have
\begin{eqnarray}\label{esingularM1}
Y_{n}\cdot 1=\delta_{n,0}1,\   \   \   \   Y^{*}_{n}\cdot 1=\delta_{n,0}1\   \   \mbox{ for }n\ge 0.
\end{eqnarray}

\begin{rmk}\label{rmk:3.1}
{\em  As we need, we here introduce an associative algebra.
Let $\mathcal{C}$ denote the associative algebra with identity over $\C$, generated by
 $a_{n},\ b_{n} \ (n\in \mathbb{Z}),$ subject to relations
\begin{eqnarray*}
a_ma_n+a_{n}a_{m}=0, \   \  \
b_mb_n+b_{n}b_{m}=0, \   \   \
a_mb_n+b_{n}a_{m}=\delta_{m+n+1,0}\   \   \mbox{ for }m,n\in \Z.
\end{eqnarray*}
Algebra $\mathcal{C}$ can be defined alternatively as the Clifford algebra of the vector space
with basis $\{ a_{n},b_n\ |\ n\in \Z\}$
equipped with the non-degenerate symmetric bilinear form $\langle\cdot,\cdot\rangle$ defined by
$$\langle a_{m},a_{n}\rangle=0=\langle b_{m},b_{n}\rangle,\   \   \langle a_{m},b_{n}\rangle=\delta_{m+n+1,0}
\   \   \mbox{ for }m,n\in \Z.$$
 It is well known that $\mathcal{C}$ has a basis consisting of vectors
\begin{eqnarray}\label{clifford-basis}
a_{m_1}\cdots a_{m_r}b_{n_1}\cdots b_{n_s}\cdot 1
\end{eqnarray}
for $r,s\ge 0,\ m_i,n_i\in \Z$ with $m_1<m_2<\cdots <m_r$ and $n_1<n_2<\cdots <n_s$.}
\end{rmk}

Next,  we use Clifford algebra $\mathcal{C}$ to give another realization of $\mathcal{A}$.
From the defining relations, it can be readily seen that  $\mathcal{C}$ admits an automorphism $\sigma$ such that
 $$\sigma(a_n)=a_{n+1},\   \   \sigma(b_n)=b_{n-1}\  \  \mbox{  for }n\in \Z.$$
Let $\mathbb{C}[\langle \sigma\rangle]$ be the group
algebra of the group $\langle \sigma\rangle$ generated by $\sigma$.
Then we have a smash product $\mathcal{C}\sharp \mathbb{C}[\langle\sigma\rangle]$, where
$\mathcal{C}\sharp \mathbb{C}[\langle\sigma\rangle]=\mathcal{C}\otimes \mathbb{C}[\langle \sigma\rangle]$ as a vector space and
 \begin{eqnarray}\label{smash-product}
 (u\otimes \sigma^{m})(v\otimes \sigma^{n})=u\sigma^{m}(v)\otimes \sigma^{m+n}
\   \   \   \mbox{ for }u,v\in \mathcal{C},\ m,n\in \Z.
\end{eqnarray}
We have:

\begin{lem}
There exists an algebra homomorphism $\pi$ from $\mathcal{A}$ to
$\mathcal{C}\sharp \mathbb{C}[\langle\sigma\rangle]$ such that
$$\pi (Y_{n})=a_{n}\otimes\sigma,\   \   \pi (Y^{*}_{n})=b_{n}\otimes\sigma^{-1}\   \   \mbox{  for }n\in \Z.$$
\end{lem}

\begin{proof} Set $A_{n}=a_{n}\otimes\sigma,\ B_{n}=b_{n}\otimes\sigma^{-1}$ for $n\in \Z$. It is straightforward to show that
\begin{eqnarray*}
&&A_mA_n+A_{n+1}A_{m-1}=0, \\
&&B_mB_n+B_{n+1}B_{m-1}=0, \\
&&A_mB_n+B_{n-1}A_{m+1}=\delta_{m+n,0}\   \   \   \mbox{ for }m,n\in \Z.
\end{eqnarray*}
For example, we have
\begin{eqnarray*}
A_mB_n+B_{n-1}A_{m+1}
&=&(a_{m}\otimes \sigma)(b_{n}\otimes \sigma^{-1})+(b_{n-1}\otimes \sigma^{-1})(a_{m+1}\otimes \sigma)\\
&=&a_{m}b_{n-1}\otimes 1+b_{n-1}a_{m}\otimes 1\\
&=&\delta_{m+n+1,0}.
\end{eqnarray*}
 Then it follows.
 \end{proof}

Using algebra homomorphism $\pi$ we obtain another PBW type basis of $\mathcal{A}$.

\begin{prop}\label{pbwbasis2}
The following vectors form a basis of $\mathcal{A}$:
\begin{eqnarray}\label{A-basis-2}
Y_{m_1}\cdots Y_{m_r}Y^{*}_{n_1}\cdots Y^{*}_{n_s}\cdot 1,
\end{eqnarray}
where $r,s\ge 0,\ m_i,\ n_i\in \Z$, satisfying the condition
$$m_1\le m_2\le \cdots \le m_r\  \mbox{ and }\  n_1\le n_2\le \cdots \le n_s.$$
\end{prop}

\begin{proof} It is clear that Clifford algebra $\mathcal{C}$ is a $\Z$-graded algebra with
$$\deg a_{n}=1 \   \mbox{ and }\   \deg b_{n}=-1\   \   \   \mbox{ for }n\in \Z.$$
For $k\in \Z$, denote by $\mathcal{C}(k)$ the homogeneous subspace of degree $k$.
We see that ${\mathcal{C}}(k)$ is linearly spanned by vectors
\begin{eqnarray}\label{clifford-basis-k}
a_{m_1}\cdots a_{m_r}b_{n_1}\cdots b_{n_s}\cdot 1
\end{eqnarray}
for $r,s\ge 0,\ m_i,n_i\in \Z$ with $r-s=k$, $m_1<m_2<\cdots <m_r$ and $n_1<n_2<\cdots <n_s$.
 From the definition of $\pi$ we have
$$\pi (\mathcal{A})\subset \sum_{k\in \Z}\mathcal{C}(k)\otimes \C \sigma^{k}\subset
\mathcal{C}\sharp \mathbb{C}[\langle\sigma\rangle].$$
In particular, for any $r,s\ge 0,\ m_1,\dots,m_r,\ n_1,\dots,n_s\in \Z$, we have
\begin{eqnarray}
&&\pi (Y_{m_1}\cdots Y_{m_r}Y^{*}_{n_1}\cdots Y^{*}_{n_s})\nonumber\\
&=&(a_{m_1}\otimes \sigma)(a_{m_2}\otimes \sigma)\cdots (a_{m_r}\otimes \sigma)
(b_{n_1}\otimes \sigma^{-1})\cdots (b_{n_s}\otimes \sigma^{-1})\nonumber\\
&=&(a_{m_1}a_{m_2+1}\cdots a_{m_r+(r-1)}\otimes \sigma^{r})(b_{n_1}b_{n_2+1}\cdots b_{n_s+(s-1)}\otimes \sigma^{-s})
\nonumber\\
&=&a_{m_1}a_{m_2+1}\cdots a_{m_r+(r-1)}b_{n_1-r}b_{n_2+1-r}\cdots b_{n_s+(s-1)-r}\otimes \sigma^{r-s}.
\end{eqnarray}
From this we see that
$$\pi(\mathcal{A})=\coprod_{n\in \Z}\left({\mathcal{C}}(n)\otimes \C \sigma^{n}\right).$$
It then follows.
\end{proof}

%Note that Proposition \ref{Abasis} can also be proved by using homomorphism $\pi$.

\begin{definition}\label{restricted}
{\em An $\mathcal{A}$-module $W$ is said to be {\em restricted} if
                    for every $w\in W,$ $Y_{n}w = 0= Y^{\ast}_{n}w$
                    for $n$ sufficiently large, or equivalently, if
                    $Y(x), Y^{\ast}(x)\in {\mathcal{E}}(W).$ }
\end{definition}

Recall that $\mathcal{A}$ is a $\Z$-graded algebra with $\deg Y_{n}=-n=\deg Y^{*}_{n}$ for $n\in \Z$, where
for $n\in \Z$, $\mathcal{A}_{n}$ denotes the homogeneous subspace of degree $n$. Set
\begin{eqnarray}
\mathcal{A}_{-}=\oplus_{n\ge 1}\mathcal{A}_{-n}.
\end{eqnarray}
Notice that $\mathcal{A}_{0}$ is a subalgebra of $\mathcal{A}$ and the subspace
$\sum_{n\ge 1}\mathcal{A}_{n}\mathcal{A}_{-n}$, denoted by $I_{0}$, is a
(two-sided) ideal of $\mathcal{A}_{0}$. Set
\begin{eqnarray}
\bar{\mathcal{A}}_{0}=\mathcal{A}_{0}/I_{0}.
\end{eqnarray}

\begin{rmk}\label{zhualgebra}
{\em We here show that there is a natural one-to-one correspondence between the set of the equivalence classes of $\N$-graded irreducible $\mathcal{A}$-modules and that of irreducible $\bar{\mathcal{A}}_{0}$-modules.
Let $W=\oplus_{n\in \N}W(n)$ be any $\N$-graded $\mathcal{A}$-module with $W(0)\ne 0$. Then
$W(0)$ is an $\mathcal{A}_{0}$-module and  $(A_{n}A_{-n})W(0)=0$ for $n\ge 1$. Consequently,
$W(0)$ is naturally an $\bar{\mathcal{A}}_{0}$-module. If $W$ (as a graded module) is irreducible,
then  $W(0)$ is an irreducible $\mathcal{A}_{0}$-module and hence an irreducible $\bar{\mathcal{A}}_{0}$-module.
On the other hand, let $U$ be an $\bar{\mathcal{A}}_{0}$-module. Then $U$ is naturally an $\mathcal{A}_{0}$-module
such that $(A_{n}A_{-n})W(0)=0$ for $n\ge 1$.
Let $\mathcal{A}_{-}$ act trivially on $U$, making $U$ an $(\mathcal{A}_{-}+\mathcal{A}_{0})$-module.
View $U$ as a $\Z$-graded $(\mathcal{A}_{-}+\mathcal{A}_{0})$-module with $U=U(0)$. Form an induced module
\begin{eqnarray}\label{Verma-module}
M_{\mathcal{A}}(U)=\mathcal{A}\otimes_{(\mathcal{A}_{-}+\mathcal{A}_{0})}U,
\end{eqnarray}
which is a $\Z$-graded $\mathcal{A}$-module. In fact, it is $\N$-graded. From definition,
the degree $0$ subspace of $M_{\mathcal{A}}(U)$ is the quotient space of $\mathcal{A}_{0}\otimes U$
by the subspace spanned by
$$ab\otimes u-a\otimes bu,\  \  xy\otimes u- x\otimes yu\ (=xy\otimes u)$$
for $a,b\in \mathcal{A}_{0},\ u\in U$ and for $x\in \mathcal{A}_{n},\ y\in  \mathcal{A}_{-n}$ with $n\ge 1$.
We see that the degree $0$ subspace of $M_{\mathcal{A}}(U)$ is
$\mathcal{A}_{0}\otimes _{\mathcal{A}_0}U=U$. If $U$ is irreducible, then $M_{\mathcal{A}}(U)$ has a unique
irreducible graded quotient module $L_{\mathcal{A}}(U)$ with $U$ as the degree $0$ subspace.}
\end{rmk}

\begin{lem}\label{omega}
Let $W$ be an $\mathcal{A}$-module and let $w\in W$ such that
$Y_{n}w=0=Y^{*}_{n}w$ for $n\ge 1$.
Then $\mathcal{A}_{-}\cdot w=0$.
\end{lem}

\begin{proof} It  suffices to show that  $\mathcal{A}_{-k}w=0$ for  $k\ge 1$.
In view of Proposition \ref{pbwbasis2}, $\mathcal{A}_{-k}$ is linearly spanned by those vectors
in (\ref{A-basis-2}) with $r,s\ge 0,\ m_i,\ n_i\in \Z$, satisfying $r+s\ge 1$,
$$m_1\le m_2\le \cdots \le m_r, \   \  n_1\le n_2\le \cdots \le n_s,$$
and $m_1+m_2+\cdots +m_r+n_1+\cdots +n_s=k$.
We now  show that $Xw=0$ for every such vector $X$ by induction on $r+s$.
If $s=0$, we must have $m_r\ge 1$, so that
$Xw=0$. Similarly, we have $Xw=0$ if $r=0$. Thus it is true for $r+s=1$.
Assume $r,s\ge 1$. If $n_s\ge 1$, we also have $Xw=0$.
Suppose $n_s\le 0$, which implies that $n_i\le 0$ for $1\le i\le s$. Then $m_r\ge 1$.
Using the defining relations of $\mathcal{A}$, we have
\begin{eqnarray}\label{induction-step}
&&Y_{m_r}Y^{*}_{n_1}Y^{*}_{n_2}\cdots Y^{*}_{n_s}=-Y^{*}_{n_1-1}Y_{m_r+1}Y^{*}_{n_2}\cdots Y^{*}_{n_s}
+\delta_{m_r+n_1,0}Y^{*}_{n_2}\cdots Y^{*}_{n_s}\nonumber\\
&=&Y^{*}_{n_1-1}Y^{*}_{n_2-1}Y_{m_r+2}Y^{*}_{n_3}\cdots Y^{*}_{n_s}
-\delta_{m_r+1+n_2,0}Y^{*}_{n_1-1}Y^{*}_{n_3}\cdots Y^{*}_{n_s}
+\delta_{m_r+n_1,0}Y^{*}_{n_2}\cdots Y^{*}_{n_s}\nonumber\\
&=&(-1)^{s}Y^{*}_{n_1-1}Y^{*}_{n_2-1}\cdots Y^{*}_{n_s-1}Y_{m_r+s}\nonumber\\
&&+\delta_{m_r+n_1,0}Y^{*}_{n_2}\cdots Y^{*}_{n_s}
+(-1)^{s-1}\delta_{m_r+s-1+n_s,0}Y^{*}_{n_1-1}Y^{*}_{n_2-1}\cdots Y^{*}_{n_{s-1}}\nonumber\\
&&+\sum_{j=1}^{s-1}(-1)^{j}\delta_{m_r+j+n_{j},0}Y^{*}_{n_1-1}\cdots Y^{*}_{n_{j}-1}Y^{*}_{n_{j+2}}\cdots Y^{*}_{n_{s}}.
\end{eqnarray}
It then follows from this and induction hypothesis  that $Xw=0$.
\end{proof}

We have the following characterization of algebra $\bar{\mathcal{A}}_{0}$:

\begin{prop}
The algebra $\bar{\mathcal{A}}_{0}$ is isomorphic to the algebra generated by $\bar{Y}_0$ and $\bar{Y}^{*}_{0}$,
subject to relation $\bar{Y}_{0}\bar{Y}^{*}_{0}=\bar{Y}^{*}_{0}\bar{Y}_{0}=1$.  Equivalently,
$\bar{\mathcal{A}}_{0}$ is isomorphic to the group algebra $\C[\Z]$ of $\Z$.
\end{prop}

\begin{proof} Set
$$\bar{Y}=Y_{0}+I_{0},\   \  \bar{Y}^{*}=Y^{*}_{0}+I_{0}\in \bar{\mathcal{A}}_{0}.$$
From the third defining relation of $\mathcal{A}$ we have
$$Y_{0}Y^{*}_{0}+Y^{*}_{-1}Y_{1}=1,\   \   \   \   Y_{-1}Y^{*}_{1}+Y^{*}_{0}Y_{0}=1.$$
It follows that $\bar{Y}\bar{Y}^{*}=\bar{Y}^{*}\bar{Y}=1$ in $\bar{\mathcal{A}}_{0}$.
Next, we prove that $\bar{\mathcal{A}}_{0}$ is generated by $\bar{Y}$ and $\bar{Y}^{*}$, which amounts to
\begin{eqnarray}\label{eA0}
\mathcal{A}_{0}=\langle Y_{0},Y^{*}_{0}\rangle+I_{0},
\end{eqnarray}
where $\langle Y_{0},Y^{*}_{0}\rangle$ denotes the subalgebra generated by $Y_{0}$ and $Y^{*}_{0}$.
Note that in view of Proposition \ref{pbwbasis2}, $\mathcal{A}_{0}$ is linearly spanned by the vectors of the form
$$X=Y_{m_1}Y_{m_2}\cdots Y_{m_r}Y^{*}_{n_1}Y^{*}_{n_2}\cdots Y^{*}_{n_s},$$
where $r,s\ge 0$, $m_i,n_i\in \Z$ with $m_1\le m_2\le \cdots \le m_r$, $n_1\le n_2\le \cdots \le n_s$, and
$$m_1+m_2+\cdots +m_{r}+n_1+n_2+\cdots +n_s=0.$$
We first consider the case with $r\ge 1,\  s=0$: $X=Y_{m_1}Y_{m_2}\cdots Y_{m_r}$.
 Since $m_1\le m_2\le \cdots \le m_r$ and $m_1+m_2+\cdots +m_r=0$, we have $m_r\ge 0$. If $m_r=0$, we have
$m_1=m_2=\cdots =m_r=0$, so that $X=Y_{0}^{r}\in \langle Y_{0},Y^{*}_{0}\rangle$. If $m_r\ge 1$, we have
$$X=(Y_{m_1}\cdots Y_{m_{r-1}})Y_{m_r}\in I_0.$$
Analogously, if $r=0,\ s\ge 1$ we have $X\in \langle Y_{0},Y^{*}_{0}\rangle$.

Now consider the case with $r,s\ge 1$. We shall use induction on $r+s$.
For $r+s=2$, we have $X=Y_{m}Y^{*}_{-m}$ with $m\in \Z$. If $m<0$, we have $X\in I_0$ by definition, and if $m\ge 0$, we have
$$Y_{m}Y^{*}_{-m}=-Y^{*}_{-m-1}Y_{m+1}+1\equiv 1\   \  \mod I_{0}.$$
Assume that $r+s\ge 2$. If $n_s\ge 1$, we have
$$X=\left(Y_{m_1}Y_{m_2}\cdots Y_{m_r}Y^{*}_{n_1}Y^{*}_{n_2}\cdots Y^{*}_{n_{s-1}}\right)Y^{*}_{n_s}\in I_0.$$
If $n_s=0$, it follows from the induction hypothesis that
$$Y_{m_1}Y_{m_2}\cdots Y_{m_r}Y^{*}_{n_1}Y^{*}_{n_2}\cdots Y^{*}_{n_{s-1}}\in \langle Y_{0},Y^{*}_{0}\rangle+I_{0},$$
so that
$$X=(Y_{m_1}Y_{m_2}\cdots Y_{m_r}Y^{*}_{n_1}Y^{*}_{n_2}\cdots Y^{*}_{n_{s-1}})Y_{0}
\in \langle Y_{0},Y^{*}_{0}\rangle+I_{0}.$$
Now assume $n_s<0$.  Then $n_1,\dots,n_s<0$. As $m_1+\cdots +m_r+n_1+\cdots +n_s=0$, we must have $m_r\ge 1$,
recalling that $m_1\le m_2\le \cdots \le m_r$.
We have (see (\ref{induction-step}))
\begin{eqnarray*}
&&Y_{m_r}Y^{*}_{n_1}Y^{*}_{n_2}\cdots Y^{*}_{n_s}\\
&=&(-1)^{s}Y^{*}_{n_1-1}Y^{*}_{n_2-1}\cdots Y^{*}_{n_s-1}Y_{m_r+s}\\
&&+\delta_{m_r+n_1,0}Y^{*}_{n_2}\cdots Y^{*}_{n_s}
+(-1)^{s-1}\delta_{m_r+s-1+n_s,0}Y^{*}_{n_1-1}Y^{*}_{n_2-1}\cdots Y^{*}_{n_{s-1}}\\
&&+\sum_{j=1}^{s-1}(-1)^{j}\delta_{m_r+j+n_{j},0}Y^{*}_{n_1-1}\cdots Y^{*}_{n_{j}-1}Y^{*}_{n_{j+2}}\cdots Y^{*}_{n_{s}}.
\end{eqnarray*}
Notice that
$$Y_{m_1}\cdots Y_{m_{r-1}}\cdot Y^{*}_{n_1-1}Y^{*}_{n_2-1}\cdots Y^{*}_{n_s-1}Y_{m_r+s}\in I_{0}$$
as $m_r+s\ge 1$. It then follows from the induction hypothesis that $X\in \langle Y_{0},Y^{*}_{0}\rangle+I_{0}.$
Therefore, $\bar{\mathcal{A}}_{0}$ is a commutative algebra generated by $\bar{Y}$ and $\bar{Y}^{*}$,
satisfying relation $\bar{Y}\bar{Y}^{*}=\bar{Y}^{*}\bar{Y}=1$.
Consequently, $\bar{\mathcal{A}}_{0}$ is a quotient algebra of the group algebra $\C[\Z]$.

Next, we show that $\bar{\mathcal{A}}_{0}$ is isomorphic to  $\C[\Z]$ by constructing enough irreducible modules.
 From Jing-Nie's realization of $\mathcal{A}$ on $M(1)$, we have $Y_{n}1=0=Y^{*}_{n}1$ for $n\ge 1$ and
$Y_{0}1=1=Y^{*}_{0}1$. By Lemma \ref{omega}, we have $\mathcal{A}_{-}\cdot 1=0$, which also implies $I_{0}\cdot 1=0$.
It then follows from (\ref{eA0}) that $\C 1$ is an irreducible $\bar{\mathcal{A}}_{0}$-module on which $\bar{Y}_{0}$  and $\bar{Y}^{*}_{0}$ act as scalar $1$. Let $\mu$ be  any nonzero complex number.
It can be readily seen that  $\mathcal{A}$ admits an automorphism $\tau_{\mu}$  such that
\begin{eqnarray}\label{edef-tau}
\tau_{\mu}(Y_{n})=\mu Y_{n} \   \mbox{ and }\   \tau_{\mu}(Y^{*}_{n})=\mu^{-1}Y^{*}_{n}\   \   \   \mbox{ for }n\in \Z.
\end{eqnarray}
Then it follows immediately that $\bar{\mathcal{A}}_{0}$ admits a $1$-dimensional module $\C_{\mu}$
on which $\bar{Y}_{0}$  and $\bar{Y}^{*}_{0}$ act as scalars $\mu$  and  $\mu^{-1}$, respectively.
 This implies that $1,\  \bar{Y}^{n},\ (\bar{Y}^{*})^{n}$ for $n\ge 1$ are linearly independent.
Consequently, we have  $\bar{\mathcal{A}}_{0}\simeq \C[\Z]$.
\end{proof}

For any  nonzero
complex number $\mu$, we have a $1$-dimensional irreducible $\bar{\mathcal{A}}_{0}$-module
$\C_{\mu}$ on which $Y_{0}$ acts as scalar $\mu$ and $Y^{*}_{0}$ acts as scalar $\mu^{-1}$.
On the other hand, since $\bar{\mathcal{A}}_{0}$ is a commutative algebra of countable dimension over $\C$,
all irreducible $\bar{\mathcal{A}}_{0}$-modules are $1$-dimensional.
Therefore, $\C_{\mu}$ for $\mu\in \C^{\times}$ exhaust the irreducible $\bar{\mathcal{A}}_{0}$-modules.
For $\mu\in \C$,
let $M_{\mathcal{A}}(\mu)$ denote the (generalized) Verma $\mathcal{A}$-module $M_{\mathcal{A}}(U)$ with $U=\C_{\mu}$ (recall (\ref{Verma-module})). Furthermore,
let $L_{\mathcal{A}}(\mu)$ be the quotient module of $M_{\mathcal{A}}(\mu)$ by
the maximal  $\N$-graded $\mathcal{A}$-submodule.

To summarize we have:

\begin{thm}\label{classification-A}
$\mathcal{A}$-modules $L_{\mathcal{A}}(\mu)$ for $\mu\in \C^{\times}$ form a complete list of
irreducible $\N$-graded $\mathcal{A}$-modules up to isomorphism.
\end{thm}

For the rest of this section, we prove that $M(1)$ is an irreducible $\N$-graded $\mathcal{A}$-module and there is a non-degenerate
symmetric bilinear form that is invariant in a certain sense. First of all,
it is straightforward to show that algebra $\mathcal{A}$ admits an involution (an order $2$ anti-automorphism) $\theta$ which is uniquely determined by
\begin{eqnarray}
\theta(Y_{n})=Y^{*}_{-n},\   \   \   \  \theta(Y^{*}_{n})=Y_{-n}\   \    \    \mbox{ for }n\in \Z.
\end{eqnarray}
For any (left) $\mathcal{A}$-module $W$,  we equip $W^{*}$ with a (left) $\mathcal{A}$-module structure given by
\begin{eqnarray}
(af)(w)=f(\theta(a)w)\    \    \   \mbox{ for }a\in \mathcal{A},\  f\in W^{*},\  w\in W.
\end{eqnarray}
A bilinear form $\<\cdot,\cdot\>$ on an $\mathcal{A}$-module $W$ is said to be {\em invariant} if
\begin{eqnarray}
 \<au,v\>=\<u,\theta(a)v\>\    \    \    \mbox{ for }a\in \mathcal{A},\ u,v\in W.
\end{eqnarray}

Recall that  $M_{\mathcal{A}}(1)$ is an $\N$-graded $\mathcal{A}$-module with $1$-dimensional degree-zero subspace $\C$.
Denote by $v_1$ the generator $1$ (an element of $\C$) of the $\mathcal{A}$-module $M_{\mathcal{A}}(1)$.
Just as with Verma modules for finite-dimensional simple Lie algebras, one obtains a symmetric invariant bilinear form $\<\cdot,\cdot\>$ on $M_{\mathcal{A}}(1)$ such that $\<v_{1},v_{1}\>=1$.
We have $\< M_{\mathcal{A}}(1)_{m},M_{\mathcal{A}}(1)_{n}\>=0$ for $m,n\in \N$ with $m\ne n$, where $ M_{\mathcal{A}}(1)_{m}$ denotes
the homogeneous subspace of degree $m$.
It then follows that any proper graded submodule is contained in the kernel of $\<\cdot,\cdot\>$,
so that $\<\cdot,\cdot\>$ is reduced  to a symmetric bilinear form on the quotient module. This gives rise to a non-degenerate
symmetric invariant bilinear form on the irreducible $\N$-graded module $L_{\mathcal{A}}(1)$.

Furthermore, we have the following results on $\mathcal{A}$-modules $L_{\mathcal{A}}(1)$ and $M(1)$:

\begin{thm}\label{tirreducible}
As $\mathcal{A}$-modules, we have
$L_{\mathcal{A}}(1)=M_{\mathcal{A}}(1)/J\simeq M(1)$, where $J$ is the $\mathcal{A}$-submodule of $M_{\mathcal{A}}(1)$,
generated by vectors
\begin{eqnarray}
Y_{-\lambda}v_1-(-1)^{|\lambda|}Y^{*}_{-\lambda'}v_1
\end{eqnarray}
for $\lambda\in \mathcal{P}$, where $\lambda'$ is the dual of $\lambda$. Furthermore,
 there is a non-degenerate symmetric invariant bilinear form $\<\cdot,\cdot\>$ on $M(1)$
with $\{Y_{-\lambda}1\ | \lambda\in \mathcal{P}\}$ as an orthonormal basis.
\end{thm}

\begin{proof} First of all, by (\ref{esingularM1}) there exists an $\mathcal{A}$-module homomorphism $\pi$ from $M_{\mathcal{A}}(1)$ onto $M(1)$ with $\pi (v_1)=1$.
 Recall that for $\lambda=(\lambda_1,\lambda_2,\dots,\lambda_r)\in \mathcal{P}$,  the weight of $\lambda$ is defined to be
$|\lambda|=\lambda_1+\lambda_2+\cdots +\lambda_r$. Then $Y_{-\lambda}v_1$ is of degree $|\lambda|$.
A simple fact is that $\lambda$ and its dual $\lambda'$ have the same weight. It then follows that  $J$ is a graded submodule.
 Set
$$\overline{M}_{\mathcal{A}}(1)=M_{\mathcal{A}}(1)/J,$$
which is an $\N$-graded $\mathcal{A}$-module with $1$-dimensional
degree-zero subspace.  Denote by $\bar{v}_{1}$  the image of $v_1$ in $\overline{M}_{\mathcal{A}}(1)$.
 In view of relation (\ref{eY=Y*}), we have $J\subset \ker \pi$, so that the $\mathcal{A}$-module homomorphism $\pi$ reduces to
a homomorphism from $\overline{M}_{\mathcal{A}}(1)$ onto $M(1)$, denoted by $\bar{\pi}$.
We next show that $\bar{\pi}$ is actually an isomorphism.

It follows from the definition of $J$ and the P-B-W basis (Proposition \ref{pbwbasis2} ) that
$\overline{M}_{\mathcal{A}}(1)$ is linearly spanned by vectors
\begin{eqnarray}
Y_{-\lambda}\bar{v}_1\  \  (\lambda\in \mathcal{P}).
\end{eqnarray}
We claim
\begin{eqnarray}\label{orthogonal}
\< Y_{-\lambda}\bar{v}_1,Y_{-\mu}\bar{v}_1\>=\delta_{\lambda,\mu}\   \   \    \mbox{ for }\lambda,\mu\in \mathcal{P}.
\end{eqnarray}
Let
$$\lambda=(\lambda_1,\lambda_2,\dots,\lambda_r), \  \  \mu=(\mu_1,\mu_2,\dots,\mu_s)\in \mathcal{P},$$
where $\lambda_i,\ \mu_j$ are positive integers such that
$$\lambda_1\ge \lambda_2\ge \cdots \ge \lambda_r\   \mbox{ and }\ \mu_1\ge \mu_2\ge \cdots \ge \mu_s.$$
Suppose $\lambda_1\ne \mu_1$. By symmetry, we may assume $\lambda_1>\mu_1$. Then we have
\begin{eqnarray*}
\<Y_{-\lambda}\bar{v}_1,Y_{-\mu}\bar{v}_1\>
&=&\<Y_{-\lambda_1}Y_{-\lambda_2}\cdots Y_{-\lambda_r}\bar{v}_1,Y_{-\mu_1}Y_{-\mu_2}\cdots Y_{-\mu_r}\bar{v}_1\>\\
&=&\<Y_{-\lambda_2}\cdots Y_{-\lambda_r}\bar{v}_1,Y^{*}_{\lambda_1}Y_{-\mu_1}Y_{-\mu_2}\cdots Y_{-\mu_r}\bar{v}_1\>\\
&=&0,
\end{eqnarray*}
noticing that as $\lambda_1+i>\mu_{i+1}$ for $i=0,\dots,s-1$,
\begin{eqnarray*}
&&Y^{*}_{\lambda_1}Y_{-\mu_1}Y_{-\mu_2}\cdots Y_{-\mu_r}\bar{v}_1
=-Y_{-\mu_1-1}Y^{*}_{\lambda_1+1}Y_{-\mu_2}\cdots Y_{-\mu_r}\bar{v}_1=\cdots\\
&&\  \  \  \  =(-1)^{s}Y_{-\mu_1-1}Y_{-\mu_2-1}\cdots Y_{-\mu_r-1}Y^{*}_{\lambda_1+s}\bar{v}_1=0.
\end{eqnarray*}
On the other hand, if $\lambda_1=\mu_1$, we have
\begin{eqnarray*}
&&\<Y_{-\lambda}\bar{v}_1,Y_{-\mu}\bar{v}_1\>\\
&=&\<Y_{-\lambda_2}\cdots Y_{-\lambda_r}\bar{v}_1,Y^{*}_{\lambda_1}Y_{-\mu_1}Y_{-\mu_2}\cdots Y_{-\mu_r}\bar{v}_1\>\\
&=&-\<Y_{-\lambda_2}\cdots Y_{-\lambda_r}\bar{v}_1,Y_{-\mu_1-1}Y^{*}_{\lambda_1+1}Y_{-\mu_2}\cdots Y_{-\mu_r}\bar{v}_1\>
+\<Y_{-\lambda_2}\cdots Y_{-\lambda_r}\bar{v}_1,Y_{-\mu_2}\cdots Y_{-\mu_r}\bar{v}_1\>\\
&=&\<Y_{-\lambda_2}\cdots Y_{-\lambda_r}\bar{v}_1,Y_{-\mu_2}\cdots Y_{-\mu_r}\bar{v}_1\>,
\end{eqnarray*}
noticing that $Y^{*}_{\lambda_1+1}Y_{-\mu_2}\cdots Y_{-\mu_r}\bar{v}_1=0$ as before.
Then (\ref{orthogonal}) follows from induction. Now, with (\ref{orthogonal}) we conclude that
$Y_{-\lambda}\bar{v}_1$ for $\lambda\in \mathcal{P}$
form a basis of $\overline{M}_{\mathcal{A}}(1)$, which is orthonormal with respect to the bilinear form $\<\cdot,\cdot\>$.
Thus $\overline{M}_{\mathcal{A}}(1)$ must be an irreducible $\N$-graded module
as any proper graded submodule is contained in the kernel of the bilinear form.  Consequently,  the homomorphism $\bar{\pi}$
from $\overline{M}_{\mathcal{A}}(1)$ onto $M(1)$ is an isomorphism, which implies that $M(1)$ is an  irreducible $\N$-graded
$\mathcal{A}$-module. Now, the proof is complete.
\end{proof}

Let  $\mu$ be a nonzero complex number. Recall from (\ref{edef-tau}) the automorphism $\tau_{\mu}$ of $\mathcal{A}$.
Denote by $\rho$ the representation of $\mathcal{A}$ on $M(1)$. As $L_{\mathcal{A}}(1)\simeq M(1)$,
it follows that $L_{\mathcal{A}}(\mu)\simeq (M(1),\rho\circ \tau_{\mu})$.
Then we immediately have:

\begin{cor}\label{cqdimention}
For any nonzero complex number $\mu$, define
$${\rm gdim}_{q}L_{\mathcal{A}}(\mu)=\sum_{n\in \N}\dim (L_{\mathcal{A}}(\mu)_{n})q^{n}.$$
Then
\begin{eqnarray}
{\rm gdim}_{q}L_{\mathcal{A}}(\mu)=\sum_{n\in \N}|\mathcal{P}_{n}|q^{n},
\end{eqnarray}
where $\mathcal{P}_{n}$ denotes the set of partitions $\lambda$ of weight $n$.
\end{cor}

\section{Associative algebra $\tilde{\mathcal{A}}$ and quantum vertex algebras}

In this section, we introduce a new associative algebra $\tilde{\mathcal{A}}$ and study its vacuum modules.
To $\tilde{\mathcal{A}}$, we associate a quantum vertex algebra $V_{\tilde{\mathcal{A}}}$,
which is proved to be an irreducible quantum vertex algebra.
Furthermore, we establish a canonical one-to-one correspondence between  restricted $\mathcal{A}$-modules
and $\phi$-coordinated $V_{\tilde{\mathcal{A}}}$-modules.

\begin{definition}
{\em Let $\tilde{\mathcal{A}}$ denote the associative algebra with identity over $\mathbb{C}$ generated by
$\tilde{Y}_n,\  \tilde{Y}_{n}^{\ast}$ $(n\in \mathbb{Z})$, subject to relations
\begin{eqnarray}
&&\tilde{Y}(z)\tilde{Y}(w)+e^{w-z}\tilde{Y}(w)\tilde{Y}(z)=0, \label{eq:3.7}\\
&&\tilde{Y}^{\ast}(z)\tilde{Y}^{\ast}(w)+e^{w-z}\tilde{Y}^{\ast}(w)\tilde{Y}^{\ast}(z)=0, \label{eq:3.8}\\
&&\tilde{Y}(z)\tilde{Y}^{\ast}(w)+e^{z-w}\tilde{Y}^{\ast}(w)\tilde{Y}(z)=w^{-1}\delta\left(\frac{z}{w}\right), \label{eq:3.9}
\end{eqnarray}
where
\begin{eqnarray}
 \tilde{Y}(z)=\sum\limits_{n\in \mathbb{Z}}\tilde{Y}_{n}z^{-n-1},\  \  \  \
  \tilde{Y}^{\ast}(z)=\sum\limits_{n\in \mathbb{Z}}{\tilde{Y}^{\ast}}_{n}z^{-n-1}.
\end{eqnarray}}
\end{definition}

\begin{rmk}\label{def:A}
{\em Note that in terms of components,  relations (\ref{eq:3.7})-(\ref{eq:3.9}) amount to
 \begin{eqnarray}
 &&\tilde{Y}_{m}\tilde{Y}_{n}+\sum\limits_{k,i\geq0}\frac{1}{k!}\binom{k}{i}(-1)^{i}\tilde{Y}_{n+k-i}\tilde{Y}_{m+i}=0, \label{eq:3.10}\\
&&\tilde{Y}^{\ast}_{m}\tilde{Y}^{\ast}_{n}+\sum\limits_{k,i\geq0}\frac{1}{k!}\binom{k}{i}(-1)^{i}\tilde{Y}^{\ast}_{n+k-i}\tilde{Y}^{\ast}_{m+i}=0, \label{eq:3.11}\\
&&\tilde{Y}_{m}\tilde{Y}^{\ast}_{n}+\sum\limits_{k,i\geq0}\frac{1}{k!}\binom{k}{i}(-1)^{i}\tilde{Y}^{\ast}_{n+i}\tilde{Y}_{m+k-i}
=\delta_{m+n+1,0} \label{eq:3.12}
 \end{eqnarray}
 for $m,n\in\mathbb{Z}$.  As the expressions on the right-hand sides are infinite sums, rigorously speaking
 $\tilde{\mathcal{A}}$ should be defined as a topological algebra.
 Also, notice that in contrast to $\mathcal{A}$,
 $\tilde{\mathcal{A}}$ is {\em not} a $\Z$-graded algebra with
$\deg \tilde{Y}_{n}=n$ and $\deg \tilde{Y}^{*}_{n}=n$ for $n\in \Z$.}
\end{rmk}

%On the other hand, the relations are closed among the generators $\tilde{Y}_{m}$ and $\tilde{Y}^{\ast}_{n}$ with $m,n\ge 0$.

\begin{definition}
{\em A nonzero vector $v$ of an $\tilde{\mathcal{A}}$-module  is called
 a {\em vacuum vector} if $\tilde{Y}_{n}v=0=\tilde{Y}^{\ast}_{n}v$ for $n\geq 0$, and
a {\em vacuum $\tilde{\mathcal{A}}$-module} is an $\tilde{\mathcal{A}}$-module  $W$
together with a vacuum vector $v$ which generates $W$ as an $\tilde{\mathcal{A}}$-module, i.e., $W=\tilde{\mathcal{A}}v$. }
\end{definition}

Next, we relate $\tilde{A}$ to $\mathcal{A}$. Let $W$ be a restricted $\mathcal{A}$-module and set $U_W=\{Y(x),Y^{\ast}(x)\}$.
From (\ref{eq:3.4})-(\ref{eq:3.6}),  we have that $U_{W}$ is an $S_{trig}$-local subset
 of ${\mathcal{E}}(W)$, noticing that $(1-z/w)\delta(z/w)=0$. Then by Theorem \ref{thm2.1},  $U_W$ generates
a weak quantum vertex algebra $\langle U_W\rangle_{e}$ in ${\mathcal{E}}(W).$

As our first result of this section we have:

\begin{prop}\label{vacuum module}
 Let $W$ be a restricted $\mathcal{A}$-module and let $\langle U_W\rangle_{e}$ be the weak quantum vertex algebra generated by the $S_{trig}$-local subset $U_W=\{Y(x),Y^{\ast}(x)\}$ of ${\mathcal{E}}(W).$
Then $\langle U_W\rangle_{e}$ is an $\tilde{\mathcal{A}}$-module with $\tilde{Y}(z)$ and $\tilde{Y}^{\ast}(z)$ acting as
 $Y_{\mathcal{E}}^{e}(Y(x),z)$ and $Y_{\mathcal{E}}^{e}(Y^{\ast}(x),z)$, respectively.
 Moreover, $(\langle U_W\rangle_{e},1_W)$ is a vacuum  $\tilde{\mathcal{A}}$-module.
\end{prop}

 \begin{proof} From  \cite{Li-phi} (Proposition 5.3), with relations (\ref{eq:3.4})-(\ref{eq:3.6}) we have
 \begin{eqnarray*}
&&Y_{\mathcal{E}}^{e}(Y(z),z_{1})Y_{\mathcal{E}}^{e}(Y(z),z_{2})=-e^{z_2-z_1}Y_{\mathcal{E}}^{e}(Y(z),z_{2})Y_{\mathcal{E}}^{e}(Y(z),z_{1}), \\
&&Y_{\mathcal{E}}^{e}(Y^{\ast}(z),z_{1})Y_{\mathcal{E}}^{e}(Y^{\ast}(z),z_{2})=-e^{z_2-z_1}Y_{\mathcal{E}}^{e}(Y(z),z_{2})Y_{\mathcal{E}}^{e}(Y(z),z_{1}),\\
&&(z_1-z_2)Y_{\mathcal{E}}^{e}(Y(z),z_{1})Y_{\mathcal{E}}^{e}(Y^{\ast}(z),z_{2})=
-(z_1-z_2)e^{z_1-z_2}Y_{\mathcal{E}}^{e}(Y^{\ast}(z),z_{2})Y_{\mathcal{E}}^{e}(Y(z),z_{1}).
\end{eqnarray*}
Furthermore, with the last identity ($\mathcal{S}$-locality) we have
\begin{eqnarray}
 && z_{0}^{-1}\delta\left(\frac{z_{1}-z_{2}}{z_{0}}\right)Y_{\mathcal{E}}^{e}(Y(z),z_{1})Y_{\mathcal{E}}^{e}(Y^{\ast}(z),z_{2})\nonumber\\
   &&\  \   \  -z_{0}^{-1}\delta\left(\frac{z_{2}-z_{1}}{-z_{0}}\right)(-e^{z_1-z_2})
   Y_{\mathcal{E}}^{e}(Y^{\ast}(z),z_{2})Y_{\mathcal{E}}^{e}(Y(z),z_{1})
                                                            \nonumber\\
&=&z_{2}^{-1}\delta\left(\frac{z_{1}-z_{0}}{z_{2}}\right)
Y_{\mathcal{E}}^{e}(Y_{\mathcal{E}}^{e}(Y(z),z_0)Y^{\ast}(z),z_{2}).\label{eq:3.13}
\end{eqnarray}
Using Lemma 6.7 of \cite{Li-phi}, together with relations (\ref{eq:3.4})-(\ref{eq:3.6}), we obtain
$Y(z)^{e}_{n}Y^{\ast}(z)=0$ for $n\geq 1$ and
\begin{eqnarray*}
Y(z)^{e}_{0}Y^{\ast}(z)
&=&\mbox{Res}_{z_{1}}\left(z^{-1}Y(z_{1})Y^{\ast}(z)+z^{-1}\frac{z_1}{z}Y^{\ast}(z)Y(z_1)\right) \nonumber\\
 &=&\mbox{Res}_{z_{1}}z^{-1}\delta\left(\frac{z_{1}}{z}\right)\nonumber\\
          &=&1\ (=1_{W}).
\end{eqnarray*}
Then applying  $\mbox{Res}_{z_{0}}$ to (\ref{eq:3.13}), we obtain
 \begin{eqnarray*}
 &&Y_{\mathcal{E}}^{e}(Y(z),z_{1})Y_{\mathcal{E}}^{e}(Y^{\ast}(z),z_{2})
+e^{z_1-z_2}Y_{\mathcal{E}}^{e}(Y^{\ast}(z),z_{2})Y_{\mathcal{E}}^{e}(Y(z),z_{1})\nonumber\\
       &=& z_{2}^{-1}\delta\left(\frac{z_{1}}{z_{2}}\right)Y_{\mathcal{E}}^{e}(Y(z)_{0}^{e}Y^{\ast}(z),z_{2})\nonumber\\
       &=& z_{2}^{-1}\delta\left(\frac{z_{1}}{z_{2}}\right).
\end{eqnarray*}
Thus, $\langle U_W\rangle_{e}$ is an $\tilde{\mathcal{A}}$-module with $\tilde{Y}(z)$ and $\tilde{Y}^{\ast}(z)$ acting as
                   $Y_{\mathcal{E}}^{e}(Y(x),z)$ and $Y_{\mathcal{E}}^{e}(Y^{\ast}(x),z)$, respectively.
                   Since the weak quantum vertex algebra $\langle U_W\rangle_{e}$ is generated by $Y(x)$ and $Y^{\ast}(x)$,
                   we see that $\langle U_W\rangle_{e}$ as an $\tilde{\mathcal{A}}$-module is generated from $1_W$.
                   With $1_{W}$ being the vacuum vector of
                   the weak quantum vertex algebra $\langle U_W\rangle_{e}$,
                   we also have  $Y(x)_{n}^{e}1_W=0=\tilde{Y}^{\ast}(x)_{n}^{e}1_W$
                   for $n\geq 0$.
 Therefore, $(\langle U_W\rangle_{e},1_W)$ is a vacuum $\tilde{\mathcal{A}}$-module.
\end{proof}

\begin{rmk}\label{rmk:3.3}
{\em Recall from  Remark \ref{rmk:3.1} the Clifford algebra $\mathcal{C}$.
Let $J_{+}$ be the left ideal generated by $a_{n}$ and $b_{n}$ for $n\ge 0$. Set
$V_{\mathcal{C}}=\mathcal{C}/J_{+}$, a left $\mathcal{C}$-module. Furthermore,  set
$${\bf 1}=1+J_{+}\in V_{\mathcal{C}}\  \mbox{  and }\ a=a_{-1}\textbf{1},\  b=b_{-1}\textbf{1}\in V_{\mathcal{C}}.$$
Then (cf. \cite{FFR})  there exists a vertex superalgebra structure on $V_{\mathcal{C}}$, which is uniquely determined by the condition that
${\bf 1}$ is the vacuum vector and
$$Y(a,x)=a(x)=\sum_{n\in \Z}a_{n}x^{-n-1}\   \mbox{ and }\  Y(b,x)=b(x)=\sum_{n\in \Z}b_{n}x^{-n-1}.$$
In fact, $V_{\mathcal{C}}$ is a simple vertex operator super-algebra.}
\end{rmk}

Let $V$ be a general vertex super-algebra.
Recall from  \cite{Li8} that a {\em pseudo-endomorphism} of $V$ is a linear map $\Delta(x): V\rightarrow V\otimes\mathbb{C}((x))$
such that
\begin{eqnarray}\label{pseudo-end}
\Delta(x){\bf 1}={\bf 1} \  \mbox{ and }\  \Delta(x)Y(u,z)v=Y(\Delta(x-z)u,z)\Delta(x)v  \   \   \mbox{ for }u,v\in V.
\end{eqnarray}

View $\C((x))$ as a vertex algebra with $1$ as the vacuum vector and with
$$Y(f(x),z)=e^{-z\frac{d}{dx}}f(x)=f(x-z)\   \   \  \mbox{ for }f(x)\in \C((x)).$$
We identify $\C$-linear maps from $V$ to $V\otimes \C((x))$ with  $\C((x))$-linear endomorphisms of $V\otimes \C((x))$.
It was proved in \cite{Li8} that a  pseudo-endomorphism of $V$, considered as a
$\C((x))$-linear endomorphism of $V\otimes \C((x))$,
exactly amounts to an endomorphism of the tensor product vertex algebra $V\otimes \C((x))$ over $\C$.
Then one can define the notion of pseudo-automorphism of $V$ in the obvious way.

For the vertex super-algebra $V_{\mathcal{C}}$ we have:

\begin{lem}\label{lvsalgebra}
For any nonzero $f(x)\in \C((x))$, there exists a pseudo-automorphism $\Delta_{f}(x)$  of $V_{\mathcal{C}}$,
which is uniquely determined by
$$\Delta_{f}(x)a=a\otimes f(x),\  \   \   \   \Delta_{f}(x)b=b\otimes f(x)^{-1}.$$
Furthermore, we have
\begin{eqnarray}\label{comm-pseudo}
\Delta_{f}(x)\Delta_{g}(z)=\Delta_{fg}(x)=\Delta_{g}(z)\Delta_{f}(x) \   \   \   \mbox{ for }f,g\in \C((x))^{\times}.
\end{eqnarray}
\end{lem}

\begin{proof} The uniqueness is clear as $V_{\mathcal{C}}$ as a vertex super-algebra is generated by $a$ and $b$.
For the rest,  we first prove that for any nonzero $f(x)\in \C((x))$, there exists a pseudo-endomorphism $\Delta_{f}(x)$  of $V_{\mathcal{C}}$ with the desired property and then prove that (\ref{comm-pseudo}) holds. Given any nonzero
$f(x)\in \C((x))$, we have
$$Y(a\otimes f(x),z)=Y(a,z)\otimes f(x-z), \  \  \  \   Y(b\otimes f(x)^{-1},z)=Y(b,z)\otimes f(x-z)^{-1}.$$
Then
\begin{eqnarray*}
&&Y(a\otimes f(x),z_1)Y(a\otimes f(x),z_2)+Y(a\otimes f(x),z_2)Y(a\otimes f(x),z_1)=0,\\
&&Y(b\otimes f(x)^{-1},z_1)Y(b\otimes f(x)^{-1},z_2)+Y(b\otimes f(x)^{-1},z_2)Y(b\otimes f(x)^{-1},z_1)=0,\\
&&Y(a\otimes f(x),z_1)Y(b\otimes f(x)^{-1},z_2)+Y(b\otimes f(x)^{-1},z_2)Y(a\otimes f(x),z_1)\\
&&\ \ =\left(Y(a,z_1)Y(b,z_2)+Y(b,z_2)Y(a,z_1)\right)\otimes f(x-z_1)f(x-z_2)^{-1}\\
&&\ \  =1\otimes z_1^{-1}\delta\left(\frac{z_2}{z_1}\right) f(x-z_1)f(x-z_2)^{-1}\\
&&\ \ =(1\otimes 1)z_1^{-1}\delta\left(\frac{z_2}{z_1}\right).
\end{eqnarray*}
It follows from the construction of $V_{\mathcal{C}}$ that there exists a homomorphism $\Phi_{f}$
of vertex super-algebras from $V_{\mathcal{C}}$ to $V_{\mathcal{C}}\otimes \C((x))$ such that
$\Phi_{f}(a)=a\otimes f(x)$ and $ \Phi_{f}(b)=b\otimes f(x)^{-1}$. Furthermore, $\Phi_{f}$ gives rise to an
endomorphism of vertex super-algebra $V_{\mathcal{C}}\otimes \C((x))$. This proves the existence.
As $a$ and $b$ generate $V_{\mathcal{C}}$ as a vertex super-algebra, (\ref{comm-pseudo}) follows immediately.
It is clear that $\Delta_{f}(x)=1$ for $f(x)=1$. Consequently, $\Delta_{f}(x)$ is a pseudo-automorphism for every nonzero $f(x)$.
\end{proof}

In view of Lemma \ref{lvsalgebra},  there exists a pseudo-automorphism $\Delta(x)$  of $V_{\mathcal{C}}$ such that
$$  \Delta(x)a=a\otimes e^{-\frac{x}{2}},\  \   \   \
\Delta(x)b=b\otimes e^{\frac{x}{2}}$$
and we have
\begin{eqnarray}\label{edelta-ab}
\Delta(x)Y(a,z)=e^{-\frac{1}{2}(x-z)}Y(a,z)\Delta(x),\   \   \   \
\Delta(x)Y(b,z)=e^{\frac{1}{2}(x-z)}Y(b,z)\Delta(x).
\end{eqnarray}
Using this we obtain the following realization of $\tilde{A}$ on $V_{\mathcal{C}}$:

\begin{prop}\label{realization-A}
There exists an $\tilde{A}$-module structure on $V_{\mathcal{C}}$ such that
\begin{eqnarray}
\tilde{Y}(x)=Y(a,x)\Delta(x),\   \    \    \  \tilde{Y}^{*}(x)=Y(b,x)\Delta(x)^{-1}.
\end{eqnarray}
Furthermore, ${\bf 1}$ is a vacuum vector of $V_{\mathcal{C}}$ viewed as an $\tilde{A}$-module.
\end{prop}

\begin{proof} Set
$$A(x)=Y(a,x)\Delta(x)\   \mbox{ and }  \   B(x)=Y(b,x)\Delta(x)^{-1},$$
acting on $V_{\mathcal{C}}$. We need to prove
\begin{eqnarray*}
&&A(z)A(w)+e^{w-z}A(w)A(z)=0, \nonumber\\
&&B(z)B(w)+e^{w-z}B(w)B(z)=0, \nonumber\\
&&A(z)B(w)+e^{z-w}B(w)A(z)=w^{-1}\delta\left(\frac{z}{w}\right).
\end{eqnarray*}
This is straightforward. For example, using (\ref{edelta-ab}) we have
\begin{eqnarray*}
&&A(z)B(w)+e^{z-w}B(w)A(z)\\
&=&Y(a,z)\Delta(z)Y(b,w)\Delta(w)^{-1}+e^{z-w}Y(b,w)\Delta(w)^{-1}Y(a,z)\Delta(z)\\
&=&Y(a,z)Y(b,w)e^{\frac{1}{2}(z-w)}\Delta(z)\Delta(w)^{-1}+Y(b,w)Y(a,z)e^{\frac{1}{2}(z-w)}\Delta(w)^{-1}\Delta(z)\\
&=&z^{-1}\delta\left(\frac{w}{z}\right)e^{\frac{1}{2}(z-w)}\Delta(z)\Delta(w)^{-1}\\
&=&z^{-1}\delta\left(\frac{w}{z}\right)e^{0}\Delta(w)\Delta(w)^{-1}\\
&=&z^{-1}\delta\left(\frac{w}{z}\right),
\end{eqnarray*}
where we are using the fact that $\Delta(w)^{-1}\Delta(z)=\Delta(z)\Delta(w)^{-1}$ and the basic delta-function substitution property.
\end{proof}

Note that Proposition \ref{realization-A} gives us a (nonzero) vacuum $\tilde{A}$-module.
Next, following \cite{Li-const} we show that there exists a vacuum $\tilde{A}$-module that is universal in the obvious sense.
First, let $T$ be the free associative algebra with identity over $\mathbb{C}$ with generators
$\bar{Y}_{n}$, $\bar{Y}^{\ast}_{n}$ for $n\in \mathbb{Z}$. Denote by $T_{+}$ the subspace of $T$ linearly spanned by the vectors
 $$a_{n_{1}}^{(1)}\cdots a_{n_{r}}^{(r)}$$ for $r\geq 1$,
 $a^{(i)}\in\{\bar{Y},\bar{Y}^{\ast}\},\  n_{i}\in \mathbb{Z}$ with $n_{1}+\cdots+n_{r}\geq 0$.
 Set $J=TT_{+}$, a left ideal of $T$, and then set $V_{T}=T/TT_{+}$, a left $T$-module.
 It is clear that  for any $w\in V_{T}$ and for $a\in\{\bar{Y},\bar{Y}^{\ast}\}$, $a_{n}w=0$ for $n$ sufficiently large.
 Then define $V_{\tilde{A}}$ to be the quotient $T$-module of $V_{T}$ by the relations (\ref{eq:3.10})-(\ref{eq:3.12}).
 One sees that $V_{\tilde{A}}$ is naturally an $\tilde{A}$-module.
 Let $\textbf{1}$ denote the image  of $1$ in  $V_{\tilde{A}}$.

 Set
\begin{eqnarray}
\tilde{a}=\tilde{Y}_{-1}\textbf{1}, \    \   \  \  \tilde{b}=\tilde{Y}^{\ast}_{-1}\textbf{1}\in V_{\tilde{A}}.
\end{eqnarray}
We have:

 \begin{prop}\label{thm3.1}
The  $\tilde{\mathcal{A}}$-module $V_{\tilde{A}}$ with vector $\textbf{1}$
is a vacuum module which is universal in the obvious sense,
and there exists a weak quantum vertex algebra structure on $V_{\tilde{A}}$,
which is uniquely determined by the condition that $\textbf{1}$ is the vacuum vector and
$Y(\tilde{a},z)=\tilde{Y}(z),\ Y(\tilde{b},z)=\tilde{Y}^{\ast}(z)$.
Furthermore, for every restricted $\tilde{A}$-module $W$, there exists a $V_{\tilde{A}}$-module structure
$Y_{W}(\cdot,z)$ on $W$, which is uniquely determined by $Y_{W}(\tilde{a},z)=\tilde{Y}(z),\ Y_{W}(\tilde{b},z)=\tilde{Y}^{\ast}(z)$.
\end{prop}

\begin{proof}
Let $H$ be a vector space with basis $\{\tilde{a}, \tilde{b}\}$ (where $\tilde{a}$ and $\tilde{b}$ are just considered as two symbols).
Define a linear map $S(x): H\otimes H\rightarrow H\otimes H\otimes {\mathbb{C}}[[x]]$ by
$$S(x)(\tilde{a}\otimes\tilde{a})=-\tilde{a}\otimes\tilde{a}\otimes e^{x},
\ \ S(x)(\tilde{b}\otimes\tilde{b})=-\tilde{b}\otimes\tilde{b}\otimes e^{x},$$
$$S(x)(\tilde{a}\otimes\tilde{b})=-\tilde{a}\otimes\tilde{b}\otimes e^{x},
\ \ S(x)(\tilde{b}\otimes\tilde{a})=-\tilde{b}\otimes\tilde{a}\otimes e^{-x}.$$
Note that $V_{\tilde{A}}$ is exactly the $(H,S)$-module $V(H,S)$ constructed in \cite{Li-const}.
From Proposition 4.3 therein,  any vacuum $\tilde{A}$-module is a homomorphism image of $V_{\tilde{A}}$.
As Proposition \ref{realization-A} provides a (nonzero) vacuum $\tilde{A}$-module,  $V_{\tilde{A}}\ne 0$,
or equivalently, ${\bf 1}\ne 0$. Thus $(V_{\tilde{A}},{\bf 1})$ is a universal vacuum $\tilde{A}$-module.
Furthermore, from Proposition 4.2 of \cite{KL},  there exists a weak quantum vertex algebra structure on
$V_{\tilde{A}}$ with $\textbf{1}$ as the vacuum vector such that $Y(\tilde{a},z)=\tilde{Y}(z),\ Y(\tilde{b},z)=\tilde{Y}^{\ast}(z)$.
The furthermore assertion on the module structure also follows.
\end{proof}

\begin{rmk}\label{match}
{\em For the weak quantum vertex algebra $V_{\tilde{A}}$, since $Y(\tilde{a},x)=\tilde{Y}(x)$
and $Y(\tilde{b},x)=\tilde{Y}^{*}(x)$, we have
\begin{eqnarray}
\tilde{a}_{n}=\tilde{Y}_{n},\   \   \   \   \tilde{b}_{n}=\tilde{Y}^{*}_{n}\    \    \mbox{ for }n\in \Z.
\end{eqnarray}
We also have
\begin{eqnarray*}
 && Y(\tilde{a},x_1)Y(\tilde{a},x_2)+e^{x_{2}-x_{1}}Y(\tilde{a},x_2)Y(\tilde{a},x_1)=0,\\
 && Y(\tilde{b},x_1)Y(\tilde{b},x_2)+e^{x_{2}-x_{1}}Y(\tilde{b},x_2)Y(\tilde{b},x_1)=0,\\
  &&Y(\tilde{a},x_1)Y(\tilde{b},x_2)+e^{x_{1}-x_{2}}Y(\tilde{b},x_2)Y(\tilde{a},x_1)
 =x_1^{-1}\delta\left(\frac{x_2}{x_1}\right),
\end{eqnarray*}
which imply  (see \cite{Li-nonlocal})
\begin{eqnarray*}
&&\tilde{a}_{n}\tilde{a}=0=\tilde{b}_{n}\tilde{b}\   \   \   \mbox{ for }n\ge 0,\\
&&\tilde{a}_{0}\tilde{b}={\bf 1} \  \mbox{ and }\  \tilde{a}_{n}\tilde{b}=0 \   \   \   \mbox{ for }n\ge 1.
\end{eqnarray*}}
\end{rmk}

Furthermore, we have:

\begin{thm}\label{thm3.2}
$V_{\tilde{A}}$ is an irreducible quantum vertex algebra.
\end{thm}

\begin{proof} Set $S=\{\tilde{a}, \tilde{b}\}\subset V_{\tilde{A}}$.
% For $n\in \mathbb{Z}$, define $\deg \tilde{Y}_{n}=-n=\deg \tilde{Y}^{\ast}_{n}$.
For any integer $n$, denote by $F_{n}$ the linear span of vectors
$$u^{(1)}_{m_1}\cdots u^{(r)}_{m_r}\textbf{1}$$
for $r\ge 1$ if $n<0$, for $r\ge 0$ if $n\ge 0$, and for $u^{(i)}\in S,\  m_i\in \Z$ with
$m_1+\cdots+m_{r}\ge -n.$ By Corollary 4.2 of \cite{Li-const},  $F_{n}=0$ for $n<0$.
From \cite{KL} (Lemma 2.9), $\{F_{n}\}_{n\in\mathbb{Z}}$ is an increasing filtration of $V_{\tilde{A}}$
such that $\textbf{1}\in F_{0}$ and
$$u_{m}F_n\subset F_{k+n-m-1}\   \   \mbox{ for } u\in F_k,\  k,m,n\in\mathbb{Z}.$$
 Consider the associated graded vector space
$\mbox{Gr}_{F}(V_{\tilde{A}})=\coprod_{n\in\mathbb{Z}}(F_n/F_{n-1})$.
From Proposition 2.10 of \cite{KL}, $\mbox{Gr}_{F}(V_{\tilde{A}})$ is a $\mathbb{Z}$-graded nonlocal vertex algebra.
Notice that $\tilde{a}=\tilde{a}_{-1}{\bf 1}\in F_1$, $\tilde{b}=\tilde{b}_{-1}{\bf 1}\in F_1$.
Set
$$\bar{a}=\tilde{a}+F_0,\   \   \   \  \bar{b}=\tilde{b}+F_0 \in F_1/F_0.$$
Since $\tilde{a}$ and $\tilde{b}$ generate $V_{\tilde{A}}$ as a nonlocal vertex algebra,
$\bar{a}$ and $\bar{b}$ generate $\mbox{Gr}_{F}(V_{\tilde{A}})$ as a nonlocal vertex algebra.
As $F_{n}=0$ for $n<0$,  $\mbox{Gr}_{F}(V_{\tilde{A}})$ is nonzero.
From relations (\ref{eq:3.10})-(\ref{eq:3.12}), we see that
$\mbox{Gr}_{F}(V_{\tilde{A}})$ is a module for the Clifford algebra $\mathcal{C}$ with $a_{n},\ b_{n}$ for $n\in \Z$ acting as
$\bar{a}_{n},\ \bar{b}_{n}$, respectively.
Consequently,  $\mbox{Gr}_{F}(V_{\tilde{A}})$ as a $\mathcal{C}$-module is isomorphic to $V_{\mathcal{C}}$.
It follows that $\mbox{Gr}_{F}(V_{\tilde{A}})$ as a nonlocal vertex algebra is isomorphic to $V_{\mathcal{C}}$.
Thus, $\mbox{Gr}_{F}(V_{\tilde{A}})$ is an irreducible module for
 $\mbox{Gr}_{F}(V_{\tilde{A}})$ viewed as a nonlocal vertex algebra.
 Then by Proposition 2.11 in \cite{KL}, $V_{\tilde{A}}$ is an irreducible module for $V_{\tilde{A}}$
 viewed as a nonlocal  vertex algebra. Therefore,  $V_{\tilde{A}}$ is an irreducible quantum vertex algebra.
\end{proof}

As an immediate consequence we have:

\begin{cor}
Every vacuum $\tilde{\mathcal{A}}$-module is irreducible and
vacuum $\tilde{\mathcal{A}}$-modules are unique up to isomorphism.
\end{cor}

As our main result we have:

\begin{thm}\label{thm3.3}
Let $W$ be a restricted $\mathcal{A}$-module.
Then there exists a $\phi$-coordinated $V_{\tilde{A}}$-module structure
$Y_{W}(\cdot,z)$ on $W$, which is uniquely determined by $Y_{W}(\tilde{a},z)=Y(z),$ $Y_{W}(\tilde{b},z)=Y^{\ast}(z)$.
On the other hand, for any $\phi$-coordinated $V_{\tilde{A}}$-module $(W,Y_W)$,
$W$ is a restricted $\mathcal{A}$-module with $Y(z)=Y_{W}(\tilde{a},z),\ Y^{\ast}(z)=Y_{W}(\tilde{b},z)$.
\end{thm}

 \begin{proof} With $W$ a restricted $\mathcal{A}$-module, from Theorem \ref{vacuum module},
 the weak quantum vertex algebra $\langle U_W\rangle_{e}$ generated by $U_W=\{Y(x),Y^{\ast}(x)\}$ is a vacuum $\mathcal{\tilde{A}}$-module with $\tilde{Y}_{n}$ and $\tilde{Y}^{*}_{n}$ acting as $Y(x)_{n}^{e}$  and $Y^{*}(x)_{n}^{e}$, respectively.
 %and $Y(x)_{n}^{e}1_{W}=0=Y^{\ast}(x)_{n}^{e}1_{W}$ for $n\geq 0$.
 Since  the vacuum module ($V_{\mathcal{\tilde{A}}}$, ${\bf 1}$)
 of $\mathcal{\tilde{A}}$ is universal, there exists an $\mathcal{\tilde{A}}$-module homomorphism
 $\rho$ from $V_{\mathcal{\tilde{A}}}$ to $\langle U_W\rangle_{e}$, sending $\textbf{1}$ to $1_{W}$. We have
 $$\rho(\tilde{a})=\rho(\tilde{Y}_{-1}{\bf 1})=Y(x)_{-1}^{e}1_{W}=Y(x), \   \  \  \
\rho(\tilde{b})=\rho(\tilde{Y}^{*}_{-1}{\bf 1})=Y^{*}(x)_{-1}^{e}1_{W}=Y^{*}(x)$$
 and
 \begin{eqnarray*}
 &&\rho(Y(\tilde{a},z)v)=\rho(\tilde{Y}(z)v)=Y_{\E}^{e}(Y(x),z)\rho(v)=Y_{\E}^{e}(\rho(\tilde{a}),z)\rho(v),\\
 &&\rho(Y(\tilde{b},z)v)=\rho(\tilde{Y}^{*}(z)v)=Y_{\E}^{e}(Y^{*}(x),z)\rho(v)=Y_{\E}^{e}(\rho(\tilde{b}),z)\rho(v)
 \end{eqnarray*}
 for $v\in V_{\tilde{A}}$.
Since $V_{\mathcal{\tilde{A}}}$ as a nonlocal vertex algebra is generated by $\tilde{a}$ and $\tilde{b}$,
it follows that $\rho$ is a homomorphism of weak quantum vertex algebras.
 Recall that $W$ is a canonical $\phi$-coordinated module for the
weak quantum vertex algebra  $\langle U_W\rangle_{e}$.
 Consequently, $W$ becomes a $\phi$-coordinated $V_{\mathcal{\tilde{A}}}$-module with
  $Y_{W}(\tilde{a},z)=Y(z)$ and $Y_{W}(\tilde{b},z)=Y^{\ast}(z)$.

On the other hand, assume that $W$ is a $\phi$-coordinated $V_{\mathcal{\tilde{A}}}$-module.
  From Propositions 5.6 and  5.9 of \cite{Li-phi}, we have
\begin{eqnarray*}
 && Y_{W}(\tilde{a},z_1)Y_{W}(\tilde{a},z_2)+\frac{z_{2}}{z_{1}}Y_{W}(\tilde{a},z_2)Y_{W}(\tilde{a},z_1)=0,\\
 && Y_{W}(\tilde{b},z_1)Y_{W}(\tilde{b},z_2)+\frac{z_{2}}{z_{1}}Y_{W}(\tilde{b},z_2)Y_{W}(\tilde{b},z_1)=0,\\
  &&Y_{W}(\tilde{a},z_1)Y_{W}(\tilde{b},z_2)+\frac{z_{1}}{z_{2}}Y_{W}(\tilde{b},z_2)Y_{W}(\tilde{a},z_1)\\
 &=& \Res_{z_{0}}z_{1}^{-1}\delta\left(\frac{z_{2}e^{z_{0}}}{z_{1}}\right)z_{2}e^{z_{0}}Y_{W}(Y(\tilde{a},z_0)\tilde{b},z_2)\\
 &=& \delta\left(\frac{z_2}{z_1}\right),
\end{eqnarray*}
where we are using  the fact that $\tilde{a}_{n}\tilde{b}=0$ for $n\ge 1$ and $\tilde{a}_{0}\tilde{b}={\bf 1}$
(see Remark \ref{match}).
Thus, $W$ is an $\mathcal{A}$-module with $Y(z)=Y_{W}(\tilde{a},z),\ Y^{\ast}(z)=Y_{W}(\tilde{b},z)$.
As $W$ is a $\phi$-coordinated $V_{\mathcal{\tilde{A}}}$-module, by the definition we have
$Y_{W}(\tilde{a},z),\ Y_{W}(\tilde{b},z)\in {\mathcal{E}}(W)$.
Therefore, $W$ is a restricted $\mathcal{A}$-module.
\end{proof}

Combining Theorems \ref{thm3.3}  and \ref{classification-A} we immediately have:

\begin{cor}
$\mathcal{A}$-modules $L_{\mathcal{A}}(\mu)$ for $\mu\in \C^{\times}$ form a complete set of
equivalence class representatives of irreducible $\N$-graded $\phi$-coordinated $V_{\tilde{A}}$-modules.
\end{cor}

\end{document}